\documentclass[11pt]{amsart}

\usepackage{amscd,amsmath,amssymb,amsfonts,verbatim}
\usepackage{graphicx}
\usepackage[active]{srcltx}
\usepackage[cmtip, all]{xy}

\newtheorem{theorem}{Theorem}[section]
\newtheorem{proposition}[theorem]{Proposition}
\newtheorem{lemma}[theorem]{Lemma}
\newtheorem{corollary}[theorem]{Corollary}

\newtheorem{question}[theorem]{Question}

\theoremstyle{definition}

\newtheorem{definition}[theorem]{Definition}

\theoremstyle{remark}
\newtheorem{remark}[theorem]{Remark}
\newtheorem*{ack}{Acknowledgments}
\numberwithin{equation}{section}

\providecommand{\Griff}{\mathop{\rm Griff}\nolimits}
\providecommand{\MU}{\mathop{\rm MU}\nolimits}

\renewcommand{\P}{{\mathbb P}}

\newcommand{\CH}{{\rm CH}}

\newcommand{\inj}{\hookrightarrow}

\newcommand{\Pic}{{\rm Pic}}
\newcommand{\Div}{{\rm Div}}

\newcommand{\Spec}{{\rm Spec \,}}

\newcommand{\NS}{{\operatorname{NS}}}

\newcommand{\Sch}{{\operatorname{\mathbf{Sch}}}}

\newcommand{\Sm}{{\mathbf{Sm}}}

\newcommand{\ds}{{/\kern-3pt/}}

\newcommand{\Tor}{{\operatorname{Tor}}}

\newcommand{\colim}{\mathop{\text{colim}}}

\newcommand{\tr}{{\operatorname{tr}}}

\renewcommand{\dim}{\text{\rm dim}}

\newcommand{\tuborg}{\left\{\begin{array}{ll}}
\newcommand{\sluttuborg}{\end{array}\right.}

\begin{document}

\title{Algebraic cobordism theory attached to algebraic equivalence}
\author{Amalendu KRISHNA and Jinhyun PARK}
\address{School of Mathematics, Tata Institute of Fundamental Research,  
1 Homi Bhabha Road, Colaba, Mumbai, 400 005, India}
\email{amal@math.tifr.res.in}
\address{Department of Mathematical Sciences, KAIST, 291 Daehak-ro, Yuseong-gu,
Daejeon, 305-701, Republic of Korea (South)}
\email{jinhyun@mathsci.kaist.ac.kr; jinhyun@kaist.edu}


\begin{abstract}Based on the algebraic cobordism theory of Levine and Morel, we develop a theory of algebraic cobordism modulo algebraic equivalence.

We prove that this theory can reproduce Chow groups modulo algebraic equivalence and the semi-topological $K_0$-groups. We also show that with finite coefficients, this theory agrees with the algebraic cobordism theory.

We compute our cobordism theory for some low dimensional varieties. The results on infinite generation of some Griffiths groups by Clemens and on smash-nilpotence by Voevodsky and Voisin are also lifted and reinterpreted in terms of this cobordism theory.

\end{abstract}


\maketitle


\section{Introduction}
The theory of algebraic cobordism  $\Omega^*$ was developed by Levine and Morel \cite{LM}. This theory is modelled on the geometric description of complex cobordism theory $\MU^*$ by Quillen \cite{Quillen}. The most interesting aspect of the algebraic cobordism theory is that it is universal among the oriented cohomology theories on smooth algebraic varieties. One of the consequences of this universality is that algebraic cobordism contains enough data to reproduce Chow groups and Grothendieck groups of algebraic varieties, as shown in \cite[Theorem 4.5.1, Corollary 4.2.12]{LM} (see also \cite{Dai}). This is in contrast with the topological situation where the natural map $\MU^* (X) \otimes_{\mathbb{L}^*} \mathbb{Z} \to H^* (X, \mathbb{Z})$ is known to be not an isomorphism in general for a CW-complex $X$ (see \cite[Theorem 2.2]{Totaro}).

More recently, Levine and Pandharipande \cite{LP} defined the double-point cobordism theory $\omega_*$ and showed that this is isomorphic to the Levine-Morel algebraic cobordism theory $\Omega_*$. As a consequence, the artificially imposed formal group law in the Levine-Morel definition of $\Omega_*$ gets a geometric interpretation. For a scheme $X$ over a field $k$, $\omega_*(X)$ is defined in terms of cobordism cycles over $X$ and the equivalence relation between these cobordism cycles is given in terms of a family of smooth cobordism cycles over $X \times \P^{1}$ which degenerate to a simple normal crossing of two smooth divisors in the family. The precise definition will be recalled below. It was remarked by Levine and Pandharipande (see \cite[\S~11.2]{LP}) that a double point cobordism theory based on algebraic equivalence should exist if one considers families of cobordism cycles parameterized by more general curves than just $\P^1$.

This leads one to the following natural question: is there  a theory of algebraic cobordism based on the Levine-Morel model, which reflects  \emph{algebraic equivalence}, reproduces Chow groups modulo algebraic equivalence and the semi-topological Grothendieck group, and more generally, interpolates between the algebraic and the complex cobordism? The goal of this paper is to develop such a theory. We go on to show that this new cobordism theory based on algebraic equivalence, also recovers the one suggested by Levine-Pandharipande in \emph{op.cit}. This was one of our motivations which led to the genesis of this paper.  

We show that the algebraic cobordism theory $\Omega_*^{\rm alg}$ interpolates between the algebraic and the complex cobordism in much the same way the semi-topological $K$-theory of Friedlander, Lawson and Walker (see \cite{FL}, \cite{FW1}, \cite{FW}) interpolates between the algebraic and the topological $K$-theories. We compute $\Omega_*^{\rm alg}$ for curves and surfaces and show that they are finitely generated modules over the Lazard ring. This is in contrast with the corresponding situation in algebraic cobordism. We further show that $\Omega_*^{\rm alg}$ agrees with $\Omega_*$ with finite coefficients. Our hope is that the functor $\Omega_*^{\rm alg}$ will inherit the properties of algebraic as well as the complex cobordism. As a consequence, this may be better suited for the study of complex algebraic varieties.  

Let $k$ always denote a fixed base field of characteristic zero throughout the paper. A \emph{scheme} in this paper will mean a separated scheme of finite type over $k$. We shall denote the category of schemes by $\Sch_k$. The full subcategory of smooth quasi-projective schemes over $k$ will be denoted by $\Sm_k$. In this paper, an oriented Borel-Moore homology theory on $\Sch_k$ and an oriented cohomology theory on $\Sm_k$ will mean the ones considered in \cite[Definitions 5.1.3, 1.1.2]{LM}.  Let $\mathbb{L}_*$ denote the Lazard ring (see \cite[p. 4]{LM}).

The central results of this paper can be summarized as follows:


\begin{theorem}\label{thm:intro main}  On the category $\Sch_k$, there are two isomorphic algebraic cobordism theories attached to algebraic equivalence : $\Omega_* ^{\rm alg}$ in Definition \ref{definition of sim cobordism} and $\omega_* ^{\rm alg}$ in Definition \ref{def:algebraic dp-cobordism}, 
which satisfy the following properties:

\emph{(1)} The functor $\Omega_* ^{\rm alg}$ defines an oriented Borel-Moore homology theory on $\Sch_k$ that respects algebraic equivalence in the sense of Definition \ref{def:respects alg equiv}. The restriction $\Omega^* _{\rm alg}$ on the subcategory $\Sm_k$, with the cohomological indexing in Definition \ref{definition of sim cobordism}, defines an oriented cohomology theory that respects algebraic equivalence in the sense of Definition \ref{def:respects alg equiv}.

\emph{(2)} $\Omega_* ^{\rm alg}$ satisfies the localization property, the $\mathbb{A}^1$-homotopy invariance and the projective bundle formula.

\emph{(3)} $\Omega^* _{\rm alg}$ is universal among all oriented cohomology theories on $\Sm_k$ that respect algebraic equivalence. Similarly, $\Omega_* ^{\rm alg}$ is universal among all oriented Borel-Moore homology theories on $\Sch_k$ that respect algebraic equivalence. 

\emph{(4)} For $X \in \Sch_k$, $\Omega_* ^{\rm alg} (X) \otimes_{\mathbb{L}_*} \mathbb{Z} \simeq \CH_* ^{\rm alg} (X)$ and $\Omega_*^{\rm alg} (X) \otimes _{\mathbb{L}_*} \mathbb{Z}[\beta, \beta^{-1}] \simeq G_0 ^{\rm semi} (X) [ \beta, \beta^{-1}]$, where $\CH_*^{\rm alg}$ is the Chow group modulo algebraic equivalence, $G_0 ^{\rm semi}$ is the semi-topological Grothendieck group of coherent sheaves, and $\beta$ is a formal symbol of degree $-1$.

\emph{(5)} For $X \in \Sch_k$, $\Omega_* (X) \otimes _{\mathbb{Z}} \mathbb{Z}/m \overset{\simeq}{\to} \Omega_* ^{\rm alg} (X) \otimes_{\mathbb{Z}} \mathbb{Z}/m$.

\end{theorem}

\begin{theorem}\label{thm:intro computation}Let $X \in \Sch_k$ and let $\mathbb{L}^*$ be the Lazard ring with the cohomological indexing. 

\emph{(1)} For $X$ smooth over $\mathbb{C}$, there is a cycle class map $\Omega^* _{\rm alg}(X) \to \MU^{2*} (X(\mathbb{C}))$.

\emph{(2)} $\mathbb{L}^* \simeq \Omega^* (k) \simeq \Omega^* _{\rm alg} (k).$ Furthermore, $\Omega_* ^{\rm alg}$ is generically constant in the sense of 
\cite[Definition 4.1.1]{LM}. 

\emph{(3)} If $X$ is a cellular scheme, then $\Omega_* (X) \overset{\simeq}{\to} \Omega_* ^{\rm alg} (X)$ as free $\mathbb{L}_*$-modules.

\emph{(4)} When $X$ is smooth, the $\mathbb{L}^*$-module $\Omega_{\rm alg} ^* (X)$ is finitely generated if and only if the group $\CH^* _{\rm alg} (X)$ is finitely generated. When $X$ is smooth projective over $\mathbb{C}$, the $\mathbb{L}^*$-module $\Omega^* _{\rm alg}(X)$ is finitely generated if and only if the Griffiths group $\Griff^* (X)$ is finitely generated. If $\dim X \leq 2$, the $\mathbb{L}^*$-module $\Omega_{\rm alg} ^* (X)$ is finitely generated, and it is not necessarily finitely generated otherwise.

\emph{(5)} If $X$ is a connected smooth affine curve, then $\mathbb{L}^* \simeq \Omega_{\rm alg} ^* (X)$. If $X$ is a smooth curve over $ \mathbb{C}$, then $\Omega^* _{\rm alg} (X)  \overset{\simeq}{\to} \MU^{2*} (X(\mathbb{C}))$ and an
analogue of Quillen-Lichtenbaum conjecture holds:
\[
\Omega^* (X) \otimes_{\mathbb{Z}} \mathbb{Z}/m \overset{\simeq}{\to} \MU^{2*} (X(\mathbb{C})) \otimes_{\mathbb{Z}} \mathbb{Z}/m.
\]

\end{theorem}

The following result of this paper is a cobordism analogue of a theorem of Voevodsky \cite{Voevodsky} and Voisin \cite{Voisin} about smash-nilpotence for algebraic cycles.

\begin{theorem}\label{thm:intro smash}Let $X$ be a smooth projective scheme and $\alpha$ be a cobordism cycle over $X$. If $\alpha$ vanishes  in $\Omega^* _{\rm alg}(X) _{\mathbb{Q}}$, then its smash-product $\alpha^{\boxtimes N}$ (see Definition \ref{def:smash-product}) vanishes in $\Omega^* (X^N)_{\mathbb{Q}}$ for some integer $N>0$. 
\end{theorem}

The organization of this paper is as follows. A good part of the paper between \S~\ref{section:LM-cobordism} and \S~\ref{section:connection to cycle and K} is devoted to proving Theorem \ref{thm:intro main}. \S~\ref{section:LM-cobordism} recalls the definition of cobordism cycles from \cite{LM}, and that of algebraic equivalence. In \S~\ref{adq cobordism}, we define $\Omega^{\rm alg} _*$ in terms of the cobordism cycles of Levine-Morel modulo various relations, one of which reflects algebraic equivalence of line bundles. We prove a universal property of $\Omega_* ^{\rm alg}$. \S~\ref{adp cobordism} recalls the rational and algebraic double-point cobordism theories $\omega_*$ and $\omega_* ^{\rm alg}$ from \cite{LP}.

Our main step in proving many of the above results is the basic exact sequence of Theorem~\ref{thm:FES}. This sequence gives an explicit relation between $\omega_*(X)$ and $\omega_* ^{\rm alg}(X)$ for any $X \in \Sch_k$. This in particular allows us to prove the isomorphism between  $\Omega_* ^{\rm alg}(X)$ and $\omega_* ^{\rm alg}(X)$ and Theorem \ref{thm:intro main}(2) in \S~\ref{section:comparison}.  \S~\ref{section:OCT} finishes the proof of Theorem \ref{thm:intro main}(1)(3) and \S~\ref{section:connection to cycle and K} proves Theorem \ref{thm:intro main}(4)(5). In \S~\ref{sec:computation}, we compute $\Omega_* ^{\rm alg}$ from various angles and prove Theorem~\ref{thm:intro computation}. In \S~\ref{sec:smash-nil}, we discuss a cobordism analogue of smash-nilpotence of algebraically trivial algebraic cycles and prove Theorem~\ref{thm:intro smash}.

Some details related to Gysin pull-backs from \cite{LM} are placed in the Appendix (\S~\ref{section:appendix}), and there a new lemma related to $\Omega_* ^{\rm alg}$ is proven.

This paper builds on two grand works \cite{LM} and \cite{LP} on algebraic cobordism. Whenever necessary, we take the freedom of using the definitions and results of these references. In doing so, we will provide full reference details.
When no confusion arises, we shall use $\sim$ to mean algebraic equivalence.

\bigskip



\section{Cobordism cycles and algebraic equivalence}\label{section:LM-cobordism}
This section recalls the basic definitions in the theory of algebraic cobordism of Levine and Morel \cite{LM}. We also recall the notion of algebraic equivalence of vector bundles and algebraic cycles. These will be used in the construction of our cobordism theory in \S~\ref{adq cobordism}.

\subsection{Cobordism cycles}

Recall the following from \cite[Definition 2.1.6]{LM}: 

\begin{definition}\label{cobordism cycle LM}\label{defn:cob-cycle-LM} Let $X \in \Sch_k$ be of dimension $n \geq 0$. An \emph{integral cobordism cycle over $X$} is a collection $(f \colon Y \to X, L_1, \cdots, L_r)$, where $Y$ is smooth and integral, $f$ is projective, and $L_1, \cdots, L_r$ ($r \ge 0$) are line 
bundles on $Y$. Its dimension is defined to be $\dim (Y) - r \in \mathbb{Z}$. An \emph{isomorphism between two cobordism cycles} $(Y\to X, L_1, \cdots, L_r) \overset{\simeq}{\to} (Y' \to X, L_1 ', \cdots, L_{r'} ')$ is a triple $\Phi = (\phi \colon Y \to Y', \sigma, (\psi_1, \cdots, \psi_r))$ consisting of an isomorphism $\phi\colon Y \to Y'$ of $X$-schemes, a bijection $\sigma\colon \{ 1, \cdots, r \} \simeq \{ 1, \cdots, r'\}$, and isomorphisms $\psi_i \colon L_i \simeq \phi^* L' _{\sigma (i)}$ of lines bundles over $Y$ for all $i$. When $Y$ is a smooth scheme with several components, define $(Y \to X, L_1, \cdots, L_r)$ to be the sum of the obvious integral cobordism cycles corresponding to the components.

Let $\mathcal{Z}_* (X)$ be the free abelian group on the set of isomorphism classes of integral cobordism cycles over $X$. We let $\mathcal{Z}_d (X)$ be the subgroup generated by the dimension $d$ cobordism cycles.
The image of the integral cobordism cycle $(Y \to X, L_1, \cdots, L_r)$ in $\mathcal{Z}_* (X)$ is denoted by $[Y \to X, L_1, \cdots, L_r]$. 
When $X$ is smooth and equidimensional, the class $[{\rm Id} _X \colon X \to X ] \in \mathcal{Z}_d (X)$ is denoted by $1_X$. A cobordism cycle of the form $[{\rm Id}_X \colon X \to X, L_1, \cdots, L_r]$ is often written as $[X \to X, L_1, \cdots, L_r]$ when no confusion arises. Recall the following definitions from \cite[2.1.2, 2.1.3]{LM}:
\end{definition}

\begin{definition}\label{basic definitions for cobordism cycles}
(1)  For a projective morphism $g\colon X \to X'$ in $\Sch_k$, the \emph{push-forward along} $g$ is the graded group homomorphism $g_* \colon \mathcal{Z}_* (X) \to \mathcal{Z}_* (X')$ given by the composition with $g$, that is, $[f\colon Y \to X, L_1, \cdots, L_r] \mapsto [ g \circ f\colon Y \to X', L_1, \cdots, L_r]$.

(2)  For a smooth equidimensional morphism $g\colon X \to X'$ of relative dimension $d$, the \emph{pull-back along} $g$ is the homomorphism $g^* \colon \mathcal{Z}_* (X') \to \mathcal{Z}_{*+d} (X)$ given by sending $[f\colon Y \to X', L_1, \cdots, L_r]$ to $[pr_2 \colon Y \times _{X'} X \to X, pr_1 ^* (L_1), \cdots, pr_1 ^* (L_r)]$.

(3) Let $L$ be a line bundle on a scheme $X$. The \emph{first Chern class operator of} $L$ is defined to be the homomorphism $\tilde{c}_1 (L)\colon \mathcal{Z}_* (X) \to \mathcal{Z}_{*-1} (X)$ given by $[f\colon Y \to X, L_1, \cdots, L_r] \mapsto [f\colon Y \to X, L_1, \cdots, L_r, f^* (L)].$ If $X$ is smooth, we define the \emph{first Chern class $c_1 (L)$ of $L$} to be the cobordism cycle $c_1 (L)\colon=[{\rm Id}_X \colon X \to X, L]$.

(4) For $X, Y \in \Sch_k$, we define the \emph{external product}
\[
\times \colon \mathcal{Z}_* (X) \times \mathcal{Z}_* (Y) \to \mathcal{Z}_* (X \times Y)
\] 
by sending the pair $[f\colon X' \to X, L_1, \cdots, L_r]\times[g\colon Y' \to Y, M_1, \cdots, M_s]$ to
\[ 
[f \times g\colon X' \times Y' \to X \times Y, pr_1 ^* (L_1), \cdots, pr_1 ^* (L_r), pr_2 ^* (M_1), \cdots, pr_2 ^* (M_s)].
\]
\end{definition}

The functor $\mathcal{Z}_* (-)$ defines the universal oriented Borel-Moore functor on $\Sch_k$ with products in the sense of \cite[Definition 2.1.10]{LM}. This universality is based on the observation in [\emph{ibid.}, Remark 2.1.8] that, in $\mathcal{Z}_* (X)$ we have the equality $[f\colon Y \to X, L_1, \cdots, L_r] = f_* \circ \tilde{c}_1 (L_r) \circ \cdots \circ \tilde{c}_r (L_1) \circ \pi_Y ^* (1),$ where $\pi_Y\colon Y \to \Spec (k)$ is the structure map and $1 \colon= 1_{\Spec (k)}\in \mathcal{Z}_0 (k)$.

\subsection{Algebraic equivalence on vector bundles}\label{sec:algebraic equivalence defn} For algebraic cycles on schemes, the notion of algebraic equivalence was defined first in \cite{Samuel}. For $X \in \Sch_k$, we say that two algebraic cycles $Z_1$ and $Z_2$ on $X$ are \emph{algebraically equivalent}, if there exists a smooth projective connected curve $C$ and $k$-rational points $t_1, t_2$ on $C$ and a cycle $Z$ on $X \times C$ such that $Z |_{X \times \{ t_j \}} = Z_j$ for $j=1,2$. We refer to \cite[Chapter 10]{Fulton} for basic facts on algebraic equivalence of algebraic cycles. For vector bundles, we have a related notion. We say two vector bundles $E_1, E_2$ of finite rank on $X$ are \emph{algebraically equivalent}, if there is a smooth projective connected curve $C$, $k$-rational points $t_1, t_2$ on $C$, and a vector bundle $V$ on $X \times C$ such that $E_i \simeq  V | _{X \times \{ t_j\}}$ for $j= 1,2$. We use $\sim_{\rm alg}$ to mean both of the above notions on cycles and vector bundles.

We say that a vector bundle $E$ of rank $m$ on $X$ is \emph{algebraically trivial} if it is algebraically equivalent to the trivial bundle $O^{\oplus m}_X$. The following facts about algebraic equivalence of vector bundles and algebraic cycles will be useful in the sequel. 

\begin{lemma}\label{lem:trivial}
Two vector bundles $E_1$ and $E_2$ on a scheme $X$ are algebraically equivalent if and only if $E_1 \otimes L$ and $E_2 \otimes L$ are algebraically equivalent for every $L \in \Pic(X)$.
\end{lemma}

\begin{lemma}\label{lem:trivial**}
Let $X$ be a smooth scheme and let $D_1$ and $D_2$ be two Weil divisors on $X$. Then $D_1 \sim_{\rm alg} D_2$ as cycles if and only if $O_X(D_1) \sim_{\rm alg} O_X(D_2)$ as line bundles.
\end{lemma}

\begin{proof}
If $D_1$ and $D_2$ are algebraically equivalent, then there is a smooth connected scheme $T$ of dimension $>0$, $k$-rational points $t_1, t_2$ on $T$, and a Weil divisor $D$ on $X \times T$ such that $D_1 - D_2 = D_{t_1} - D_{t_2}$. We can assume that $T$ is projective.  

By \cite[Theorem 1]{Kleiman} (see also \cite[Example 10.3.2]{Fulton} if $k$ is algebraically closed), we can replace $T$ by a smooth projective curve $C$ passing through $t_1, t_2$. Thus, we have a Weil divisor $D$ on $X \times C$ such that $D_1 - D_2 = D_{t_1} - D_{t_2}$. We can modify $D$ by $D- (D_{t_2} \times C) + (D_2 \times C)$ so that $D_{t_i} = D_i$ for $i =1, 2$. Letting $\mathcal{L} = O_{X \times C}(D)$, we see that $\mathcal{L} |_{X \times \{t_i\}} \simeq O_X(D_i)$ for $i = 1,2$.

Conversely, suppose there is a line bundle $\mathcal{L}$ on $X \times C$ such that $\mathcal{L} |_{X \times \{t_i\}} \simeq O_X(D_i)$ for $i = 1,2$. Since $X\times C$ is smooth, there is a Weil divisor $D$ on $X \times C$ whose associated line bundle is $\mathcal{L}$. This implies in particular that $D_{t_i} \sim_{\rm rat} D_i$ for $i = 1,2$. In other words, we have $D_1  \sim_{\rm rat} D_{t_1} \sim_{\rm alg} D_{t_2} \sim_{\rm rat} D_{2}$, which implies that $D_1 \sim_{\rm alg} D_2$.
\end{proof}

\begin{remark}\label{remk:rat-iso}
Note that if the curve $C$ in the above definition is (a nonempty open subset of) $\mathbb{P}^1$, then we can say that $E_1$ and $E_2$ are rationally equivalent. However, when $X$ is  semi-normal, this is equivalent to saying that $E_1$ and $E_2$ are isomorphic.\end{remark}

\section{The algebraic cobordism $\Omega_* ^{\rm alg}$ modulo algebraic equivalence}\label{adq cobordism}
In this section, we define an algebraic cobordism theory of a scheme $X$ associated to algebraic equivalence.
The starting point is the simple observation that the algebraic cobordism $\Omega_* (X)$ is associated to the rational equivalence of line bundles in that, two line bundles on a smooth scheme are isomorphic if and only if they are rationally equivalent (see Remark \ref{remk:rat-iso}).

Levine and Morel constructed $\Omega_* (X)$ from the 
cobordism cycles $\mathcal{Z}_* (X)$ of Definition \ref{defn:cob-cycle-LM}. 
We will define the cobordism theory $\Omega^{\rm alg} _* (X)$ that is similar to 
that of Levine-Morel, with one additional set of relations that identifies two 
integral cobordism cycles when their line bundles are suitably related by 
algebraic equivalence. 
Recall our notation of using $\sim$ for algebraic equivalence.

\begin{definition}[{Compare with \cite[Definition 2.4.5]{LM}}]\label{precobordism definition}For $X \in \Sch_k$, the \emph{$\sim$-pre-cobordism} ${\underline{\Omega}}^{\rm alg}_*  (X)$ is the quotient of $\mathcal{Z}_* (X)$ by the following three relations:

(1) (Dim) If there is a smooth quasi-projective morphism $\pi\colon Y \to Z$ with line bundles $M_1, \cdots, M_{s>\dim  Z}$ on $Z$ with $L_i \simeq \pi^* M_i$ for $i=1, \cdots , s \leq r$, then $[f\colon Y \to X, L_1, \cdots, L_r ] = 0.$

(2) (Sect) For a section $s\colon Y \to L$ of a line bundle $L$ on $Y$ with its smooth associated divisor $i\colon D \to Y,$ we impose
\[
[f\colon Y \to X, L_1, \cdots, L_r, L] = [f \circ i \colon D \to X, i^* L_1, \cdots, i^* L_r].\]

(3) (Equiv) $[Y \to X, L_1, \cdots, L_r]$ and $[Y' \to X, L_1 ', \cdots, L_r ']$ are identified if there exists an isomorphism $\phi\colon Y \to Y'$ over $X$, a permutation $\sigma$ of $\{ 1, \cdots, r\}$ and algebraic equivalences of the line bundles $L_i \sim \phi^* (L' _{\sigma (i)})$.
\end{definition}

It is immediate from the definition that there is a natural surjection 
$\underline{\Phi}_X \colon \underline{\Omega}_* (X) \to \underline{\Omega}_* ^{\rm alg} (X)$.

\begin{remark}If we take the quotient of $\mathcal{Z}_* (X)$ by only the relations (Dim) and (Sect), then the resulting quotient group is the pre-cobordism $\underline{\Omega}_* (X)$ of Levine-Morel in \emph{ibid.}
The cobordism cycles of the form $[Y \to X, L] - [Y \to X, L']$ are zero in $\underline{\Omega}_*(X)$ if $L \simeq L'$. If $\sim$ in (Equiv) is replaced by the rational equivalence $\sim_{\rm rat}$ of line bundles, then by Remark \ref{remk:rat-iso}, the modified relation ${\rm (Equiv)}_{{\rm rat}}$ plays no role because $Y$ is smooth, thus semi-normal. 
\end{remark}

\begin{lemma}\label{operations for precobordism}All four operations (projective push-forward, smooth pull-back, external product and the first Chern class operation) in Definition \ref{basic definitions for cobordism cycles} descend onto the $\sim$-pre-cobordism $\underline{\Omega}_* ^{\rm alg} $.

\end{lemma}
\begin{proof}By \cite[Remarks 2.1.11, 2.1.14, Lemmas 2.4.2, 2.4.7]{LM}, $\underline{\Omega}_* $ is an oriented Borel-Moore functor on $\Sch_k$ with product in 
the sense of \cite[Definition 2.1.10]{LM}. This implies that (Dim) and (Sect) are respected by the four operations. For (Equiv), it follows from the fact that the 
pull-back operations on line bundles via any morphisms respect algebraic equivalence.
\end{proof}

To impose the formal group law into our cobordism theory as in \cite[p. 4, \S 2.4.4]{LM}, first recall  from \emph{ibids.} that there is a graded polynomial ring $\mathbb{Z}[a_{i,j}| i,j \ge 0]$, where $a_{i,j}$ are variables of degree $i+j-1$ subject to some relations. This is called the Lazard ring, written $\mathbb{L}_*$. There is a power series $F_{\mathbb{L}_*}(u,v) \colon= \sum_{i,j} a_{i,j} u^i v^j \in {\mathbb{L}}_* [[u,v]]$ such that the pair $(\mathbb{L}_*, F_{\mathbb{L}_*})$ is the universal commutative formal group law of rank one. One also uses the cohomological indexing $\mathbb{L}^*$ by letting $\mathbb{L}^n = \mathbb{L}_{-n}$. We have $\mathbb{L}^0 \simeq \mathbb{Z}$ and $\mathbb{L}^{-n} = \mathbb{L}_n =0$ if $n<0$. Now we define the main object of study of this paper.

\begin{definition}[{Compare with \cite[Definition 2.4.10]{LM}}]\label{definition of sim cobordism}\label{defn:st cobordism} For $X \in \Sch_k$, the graded group $\Omega_* ^{\rm alg} (X)$ is defined to be the quotient of $\mathbb{L} _*\otimes_{\mathbb{Z}} \underline{\Omega}_* ^{\rm alg} (X)$ by the relations (FGL) of the form $F_{\mathbb{L}_*} (\tilde{c}_1 (L) , \tilde{c}_1 (M))([f\colon Y \to X, L_1, \cdots, L_r] )= \tilde{c}_1 (L \otimes M)([f\colon Y \to X, L_1, \cdots, L_r])$ for lines bundles $L$ and $M$ on $X$. By the relation (Dim) in Definition \ref{precobordism definition}-(1), the expression $F_{\mathbb{L}_*} (\tilde{c}_1 (L), \tilde{c}_1 (M))$ is a finite sum so that the operator is well-defined. This graded abelian group $\Omega_* ^{\rm alg}(X)$ is called the \emph{algebraic cobordism of $X$ modulo algebraic equivalence}.

When $X$ is smooth and equidimensional of dimension $n$, the codimension of a cobordism $d$-cycle is defined to be $n-d$. We set 
$\Omega^{n-d} _{\rm alg} (X)\colon = \Omega_{d} ^{\rm alg} (X)$, and 
$\Omega^* _{\rm alg} (X)$ is the direct sum of the groups over the all 
codimensions.
\end{definition}

If we omit the relation (Equiv) in the above process,
we obtain the algebraic cobordism theory $\Omega_* (X)$ of \cite[Definition 2.4.10]{LM}. In particular, we have a natural surjection $\Phi_X\colon \Omega_* (X) \to \Omega_* ^{\rm alg} (X)$.
We immediately see the following: 

\begin{proposition}\label{operations for sim cobordism}All four operations (projective push-forward, smooth pull-back, exterior product, and the first Chern class operation) in Definition \ref{basic definitions for cobordism cycles} descend onto the cobordism ${\Omega}_* ^{\rm alg}$.
\end{proposition}

\begin{remark}
By definition, we have a natural ring homomorphism 
\begin{equation}\label{eqn:base}
\Phi^{\rm alg}  \colon \mathbb{L}_*  \to \Omega_*^{\rm alg} (k)
\end{equation}
induced from the quotient map $ \mathbb{L}_* \otimes_{\mathbb{Z}} \underline{\Omega}_* ^{\rm alg} (k) \to \Omega_* ^{\rm alg} (k)$, which factors through the known map $\Phi\colon \mathbb{L}_* \to \Omega_*(k)$ in \cite[p.39]{LM}. We will see later in Proposition \ref{point case} that this is an isomorphism.
\end{remark}

We have a natural map $ q^{\rm alg} \colon {\underline{\Omega}}^{\rm alg}_*  (X) \to \Omega_* ^{\rm alg} (X)$. 
It was proven in  \cite[Lemma 2.5.9]{LM} that the  map 
$q\colon \underline{\Omega}_* (X) \to \Omega_* (X)$ is surjective.
We have a similar result:

\begin{lemma}\label{lem:surj}
Given any scheme $X$, the abelian group $\Omega_* ^{\rm alg} (X)$ is generated by the images of the integral cobordism cycles $[Y \to X, L_1, \cdots, L_r]$. In other words, the natural map $\mathcal{Z}_*(X) \to \Omega_* ^{\rm alg} (X)$ is surjective.
\end{lemma}

\begin{proof}
It suffices to show that the map $q^{\rm alg}\colon {\underline{\Omega}}^{\rm alg}_* \to \Omega_* ^{\rm alg} (X)$ is surjective. But, this follows from the observation that in the commutative diagram
\begin{equation}
\begin{CD}
\underline{\Omega}_* (X) @>{\underline{\Phi}_X}>> 
\underline{\Omega}_* ^{\rm alg} (X) \\
@V{q}VV @V{q^{\rm alg}}VV\\
\Omega_* (X) @>{\Phi}_X>> \Omega_* ^{\rm alg} (X),
\end{CD}
\end{equation} 
the map $\Phi_X$ is clearly surjective and $q$ is surjective by 
\cite[Lemma 2.5.9]{LM}. 
\end{proof}

Our discussion so far summarizes as follows (compare with \cite[Theorem 2.4.13]{LM}):

\begin{proposition}The theory $\Omega_* ^{\rm alg}$ is an oriented Borel-Moore $\mathbb{L}_*$-functor on $\Sch_k$ of geometric type in the sense of {\rm \cite[Definitions 2.1.2, 2.1.12, 2.2.1]{LM}}. 
\end{proposition}

In the rest of this section, we shall prove the following
universal property of $\Omega_* ^{\rm alg}$.

\begin{definition}\label{def:respects alg equiv}Let $A_*$ be an oriented Borel-Moore $\mathbb{L}_*$-functor on $\Sch_k$ of geometric type. We say that \emph{$A_*$ respects algebraic equivalence}, if for any $X \in \Sch_k$ and for any pair of algebraically equivalent line bundles $L$ and $M$ over $X$, we have $\tilde{c}_1 (L) = \tilde{c}_1 (M)$ as operators $A_* (X) \to A_{*-1} (X)$. 

We say that an oriented cohomology theory $A^*$ on $\Sm_k$ \emph{respects algebraic equivalence}, if for $X \in \Sm_k$ and a pair of algebraically equivalent line bundles $L$ and $M$ over $X$, we have $\tilde{c}_1 (L) = \tilde{c}_1 (M)$ as operators $A^* (X) \to A^{*+1} (X)$. 
\end{definition}


\begin{proposition}\label{prop:universal lev0}The theory $\Omega_* ^{\rm alg}$ is universal among all oriented Borel-Moore $\mathbb{L}_*$-functors on $\Sch_k$ of geometric type that respect algebraic equivalence. In other words, for any theory $A_*$ satisfying Definition \ref{def:respects alg equiv}, there exists a unique morphism $\theta_A \colon \Omega_* ^{\rm alg} \to A_*$ of oriented Borel-Moore $\mathbb{L}_*$-functor of geometric type on $\Sch_k$. 
\end{proposition}

We shall prove this proposition following a series of deductions. These intermediate
results provide useful information on the relationship between $\Omega_*$ and $\Omega_* ^{\rm alg}$.

\begin{lemma}\label{lem:Alg-st-elem}
The kernel of the map $\underline{\Phi}_X \colon \underline{\Omega}_*(X) \to \underline{\Omega}^{\rm alg}_*(X)$ is a subgroup generated by elements of the form $[f\colon Y \to X, L] -[f\colon Y \to X, M]$ with $L \sim M$.
\end{lemma}
\begin{proof}
Let $\theta_X\colon \underline{\Omega}_*(X) \twoheadrightarrow \overline{\underline{\Omega} }^{\rm alg} _*(X)$ be the quotient of $\underline{\Omega}_*(X)$ by the subgroup generated by elements given in the lemma. It follows from the definition and the surjection $\mathcal{Z}_*(X) \twoheadrightarrow \underline{\Omega}_*(X)$ that $\ker( \underline{\Phi}_X)$ is generated by elements of the form $\eta = [f\colon Y \to X, L_1, \cdots , L_r] -[f'\colon Y' \to X, L'_1, \cdots , L'_r],$ where $\phi\colon Y \to Y'$ is an isomorphism over $X$ and $\sigma $ is a permutation of $\{ 1, \cdots, r \}$ such that $L_i \sim \phi^*(L' _{\sigma (i)})$. It suffices to show that such elements vanish in $\overline{\underline{\Omega}}^{\rm alg}_*(X)$. We can modify $\eta$ so that $\eta = [f\colon Y \to X, L_1, \cdots , L_r] - [f\colon Y \to X, L'_1, \cdots , L'_r]$, where $L_i \sim L'_i$ for $1 \le i \le r$ by virtue of the relations in $\underline{\Omega}_*(X)$ as described in Definition \ref{defn:cob-cycle-LM}. Since $\theta_X(\eta)  =   f_* \circ \theta_Y \{\tilde{c}_1 (L_1) \circ \cdots \circ \tilde{c}_1 (L_r)(1_Y) - \tilde{c}_1 (L'_1) \circ \cdots \circ \tilde{c}_1 (L'_r)(1_Y)\}, $ it is enough to consider the case when $X = Y$ and $f = {\rm Id}_Y$. The lemma now follows by repeated applications of the Chern class operators, \emph{i.e.}, $\tilde{c}_1 (L_1) \circ \cdots \circ \tilde{c}_1 (L_r)(1_Y) = \tilde{c}_1 (L'_1) \circ \cdots \circ \tilde{c}_1 (L'_r)(1_Y)$ in $\overline{\underline{\Omega}}^{\rm alg}_*(Y)$.
\end{proof}

For $X \in \Sch_k$,  let $\tilde{\mathcal{R}}^{\rm alg}_* (X)$ denote (compare with \cite[Definition 2.5.13]{LM}) the graded subgroup of $\underline{\Omega}^{\rm alg}_*(X)$ generated by elements of the form
\begin{equation}\label{eqn:Cob-reln} f_* \circ \tilde{c}_1 (L_1) \circ \cdots \circ \tilde{c}_1 (L_r) \{ F\left(\tilde{c}_1 (L), \tilde{c}_1 (M)\right)(\eta) - \tilde{c}_1(L \otimes M)(\eta)\},
\end{equation}
where $[f\colon Y \to X, L_1, \cdots , L_r]$ is a standard cobordism cycle, $L, M \in \Pic(Y)$ and $\eta \in \underline{\Omega}^{\rm alg}_*(Y)$.
Since we have a natural surjection $\mathcal{Z}_*(k) \twoheadrightarrow \Omega_*(k)$ (see \cite[Lemma 2.5.9]{LM}) and the isomorphism $\Phi\colon \mathbb{L}_* \xrightarrow{\simeq} \Omega_*(k)$ (see \cite[Theorem 1.2.7]{LM}), each element $a_{i,j} \in \mathbb{L}_*$ has a lift in $\mathcal{Z}_*(k)$. In particular, the elements of the form $F\left(\tilde{c}_1 (L), \tilde{c}_1 (M)\right)(\eta)$ are well-defined in $\underline{\Omega}_*(Y)$, thus well-defined in $\underline{\Omega}^{\rm alg}_*(Y)$. Set $\tilde{\Omega}^{\rm alg}_*(X) \colon = {\underline{\Omega}^{\rm alg}_*(X)}/{\tilde{\mathcal{R}}^{\rm alg}_* (X)}$.
The following result is a refinement of Lemma \ref{lem:surj}.

\begin{proposition}\label{prop:generators}
For any $X \in \Sch_k$, there is a natural map $\psi^{\rm alg}_X \colon \tilde{\Omega}^{\rm alg}_*(X) \to \Omega^{\rm alg}_*(X)$ which is an isomorphism.
\end{proposition}

\begin{proof}
It follows from Definition \ref{defn:st cobordism} that the map $\underline{\Omega}^{\rm alg}_*(X) \to \Omega^{\rm alg}_*(X)$ kills $\tilde{\mathcal{R}}^{\rm alg}_* (X)$. This induces the natural map $\psi^{\rm alg}_X \colon  \tilde{\Omega}^{\rm alg}_*(X) \to \Omega^{\rm alg}_*(X)$. We have already shown in Lemma \ref{lem:surj} that this map is surjective. We define an inverse $\phi^{\rm alg}_X \colon \Omega^{\rm alg}_*(X) \to \tilde{\Omega}^{\rm alg}_*(X)$ of $\psi^{\rm alg}_X$ to complete the proof of the proposition.

To do this, we consider the commutative diagram
\begin{equation}\label{eqn:generators1}
\xymatrix@C.9pc{
\underline{\Omega}_*(X) \ar@{->>}[rr] \ar@{->>}[dd] \ar[dr] & & \underline{\Omega}^{\rm alg}_*(X) 
\ar@{->>}[dd] \ar[dr] & \\
& \mathbb{L}_* \otimes \underline{\Omega}_*(X) \ar@{->>}[rr]|->>>>>{\overline{\beta}}  \ar@{->>}[dd]^<<<<<<<{\alpha} \ar@{-->}[dr]|-{\gamma}& & 
\mathbb{L}_* \otimes \underline{\Omega}^{\rm alg}_*(X) \ar@{->>}[dd]|-{\alpha^{\rm alg}}  \ar@{-->}[dl]|-{\gamma^{\rm alg}}\\
\tilde{\Omega}_*(X)\ar@{->>}[rr] \ar@<1ex>[dr]|-{\psi_X} & & \tilde{\Omega}^{\rm alg}_*(X) 
\ar@<1ex>[dr]^{\psi^{\rm alg}_X} & \\
 & \Omega_*(X) \ar@{->>}[rr]_{\beta} \ar@<1ex>[ul]|-{\phi_X} & & \Omega^{\rm alg}_*(X),
\ar@<1ex>@{-->}[ul] |-{\phi_X ^{\rm alg}}}
\end{equation}where $\tilde{\Omega}_* (X)$ is defined in \cite[Definition 2.5.13]{LM}.
All the squares in the above diagram commute and the maps $\psi_X$ and $\phi_X$ are inverses of each other by \cite[Proposition 2.5.15]{LM}. By Lemma \ref{lem:Alg-st-elem}, the kernel of the map $\overline{\beta}$ is generated by elements of the form $a \otimes \left([Y \to X, L]- [Y \to X, M]\right)$, where $L \sim M$ and $a \in \mathbb{L}_*$. On the other hand, such an element maps to $\Phi(a)\left([Y \to X, L]- [Y \to X, M]\right)$ in $\tilde{\Omega}_*(X)$ under $\phi_X \circ \alpha$ (see \eqref{eqn:base}). In particular, these elements are killed in $\tilde{\Omega}^{\rm alg}_*(X)$ under the composite map $\gamma \colon \mathbb{L}_* \otimes \underline{\Omega}_*(X) \to \Omega_*(X) \to \tilde{\Omega}_*(X) \to \tilde{\Omega}^{\rm alg}_*(X)$. Thus it descends to the quotient $\gamma ^{\rm alg} \colon \mathbb{L}_* \otimes \underline{\Omega}^{\rm alg}_*(X) \to  \tilde{\Omega}_* ^{\rm alg}(X)$.

Next, we see from Definition \ref{defn:st cobordism} that the kernel of $\alpha^{\rm alg}$ is generated by elements of the form $F_{\mathbb{L}_*} (\tilde{c}_1 (L) , \tilde{c}_1 (M))([f\colon Y \to X, L_1, \cdots, L_r]) - \tilde{c}_1 (L \otimes M)([f\colon Y \to X, L_1, \cdots, L_r])$ for line bundles $L_i$ on $Y$, and line bundles $L$ and $M$ on $X$. But these elements also lie in the kernel of the map 
$\alpha$. In particular, they die in $\tilde{\Omega}^{\rm alg}_*(X)$ via 
$\gamma$ from which we conclude that $\ker (\alpha^{\rm alg}) \subseteq \ker(\gamma^{\rm alg})$. Hence, the map $\gamma^{\rm alg}$ descends to $\phi^{\rm alg} _X \colon \Omega^{\rm alg}_*(X) \to \tilde{\Omega}^{\rm alg}_*(X)$ which makes all the squares commute. It is clear from the construction that $\phi^{\rm alg}_X \circ \psi^{\rm alg}_X$ is the identity map. In particular, $\psi^{\rm alg}_X$ is injective, thus an isomorphism.
\end{proof}

\begin{proposition}\label{prop:Alg-st}
For $X \in \Sch_k$, the kernel of the natural surjection $\Phi_X \colon \Omega_*(X) \to \Omega^{\rm alg}_*(X)$ is the graded subgroup generated by the cobordism cycles of the form $[f\colon Y \to X, L]-[f\colon Y\to X, M]$, where $L$ and $M$ are algebraically equivalent.
\end{proposition}

\begin{proof}
In the commutative diagram
\begin{equation}\label{eqn:Alg-st1}
\xymatrix@C.9pc{0 \ar[r] & \tilde{\mathcal{R}}_* (X) \ar[r] \ar[d] & \underline{\Omega}_*(X) \ar[r] \ar[d]^{\underline{\Phi}_X} & \Omega_*(X) \ar[r] \ar[d]^{\Phi_X} & 0 \\
0 \ar[r] & \tilde{\mathcal{R}}^{\rm alg}_* (X) \ar[r] & \underline{\Omega}^{\rm alg}_*(X) \ar[r] & \Omega^{\rm alg}_*(X) \ar[r] & 0,}
\end{equation}
the top row is exact by \cite[Proposition 2.5.15]{LM} and the bottom row is exact by Proposition \ref{prop:generators}. The left vertical arrow in this diagram is surjective by the definition of $\tilde{\mathcal{R}}^{\rm alg}_* (X)$ above and that of $\tilde{\mathcal{R}}_* (X)$ in \cite[Lemma 2.5.14]{LM}. Hence, the map $\ker \left(\underline{\Phi}_X\right) \to \ker \left(\Phi_X\right)$ is surjective by the snake lemma. On the other hand, Lemmas \ref{lem:trivial} and 
\ref{lem:Alg-st-elem} imply that the group $\ker \left( \underline{\Phi}_X \right)$ is generated by cobordism cycles of the form $[f\colon Y \to X, L]- [f\colon Y \to X, M]$ where $L \sim M$. This proves the proposition.  
\end{proof}


{\bf{Proof of Proposition~\ref{prop:universal lev0}}:}
The theory $\Omega_* ^{\rm alg}$ satisfies Definition \ref{def:respects alg equiv} in view of the relation (Equiv) of Definition \ref{precobordism definition}. To prove its universality, we first recall from \cite[Theorem 2.4.13]{LM} that 
the algebraic cobordism $\Omega_*$ of Levine-Morel is a universal oriented Borel-Moore $\mathbb{L}_*$-functor of geometric type. So, there is a morphism $\theta  \colon \Omega_* \to A_*$ of oriented Borel-Moore $\mathbb{L}_*$-functor of geometric type on $\Sch_k$. 

To show that it induces $\theta_A \colon \Omega_* ^{\rm alg} \to A_*$, 
it is enough to show using Proposition \ref{prop:Alg-st} that $\theta (\eta) = 0$ in $A_*(X)$ for $\eta \colon= [f\colon Y \to X, L] - [f\colon Y \to X, M]$, where $L \sim M$. This is equivalent to 
$f_* \left(\left(\tilde{c}_1(L) - \tilde{c}_1 (M) \right) (1_Y) \right) = 0 \in A_* (X).$ But this holds by the assumption that $\tilde{c}_1 (L) = \tilde{c}_1 (M)$ on $A_* (Y)$. Hence, we have the induced morphism $\theta_A \colon \Omega_* ^{\rm alg} \to A_*$ as desired.
$\hspace*{12.8cm} \hfil \square$

Some fundamental properties of $\Omega^{\rm alg}_*$ will be studied in \S~\ref{subsection :Fundamental-STC} and \S~\ref{section:OCT}. We shall also show that $\Omega_* ^{\rm alg}$ is an oriented cohomology theory on $\Sm_k$ and an oriented Borel-Moore homology theory on $\Sch_k$ (see \cite[Definitions 1.1.2, 5.1.3]{LM}), equipped with a similar universal property.

\section{Algebraic double-point cobordism $\omega_* ^{\rm alg}$}\label{adp cobordism}
In this section, we recall the cobordism theory $\omega_*$ of \cite{LP} based on the double-point relations and study its \emph{algebraic equivalence} analogue $\omega_* ^{\rm alg}$ following the suggestion of  Levine and Pandharipande 
in [\emph{ibid.}, \S11.2]. 


\subsection{Double-point cobordism after Levine-Pandharipande}\label{subsection:LP}
The description of $\omega_*$ by Levine and Pandharipande is simpler than that of Levine and Morel's $\Omega_*$ in \cite{LM} in the following respects: first, the cobordism cycles are simpler, \emph{i.e.}, without the attached line bundles and the artificial imposition of the formal group law as in Definitions \ref{cobordism cycle LM} and \ref{defn:st cobordism}, and second, the relations are given by a single sort of morphisms called \emph{double-point degenerations}.

The cobordism cycles in the sense of Levine-Pandharipande, recalled below, will also be called cobordism cycles whenever no confusion arises. 

\begin{definition}[{\cite[\S 0.2]{LP}}] Let $X \in \Sch_k$. An \emph{integral cobordism cycle} on $X$ is the isomorphism class over $X$ of a projective morphism $f\colon Y \to X$, where $Y$ is smooth and integral. This will be denoted by $[f\colon Y \to X]$. Its dimension is by definition $\dim Y$. If $Y = \coprod Y_i \in \Sm_k$ where each $Y_i$ is integral, then given a projective morphism $f\colon Y \to X$, the cobordism cycle $[f\colon Y \to X]$ is defined to be the sum of $[f|_{Y_i}\colon Y_i \to X]$. Let $\mathcal{M}_* (X)^+$ be the free abelian group on the set of all integral cobordism cycles over $X$, and let $\mathcal{M}_d (X)^+$ be its subgroup generated by the cobordism cycles of dimension $d$. An element of  $\mathcal{M}_* (X)^+$ will be a called \emph{cobordism cycle}.
\end{definition}

Now we recall the notion of double-point degenerations and the associated relations from \cite[\S 0.2,  \S 0.3, \S11.2]{LP}.

\begin{definition}\label{double point degeneration} Let $Y \in \Sm_k$ be of pure dimension. Let $(C, p)$ be a pair consisting of a smooth projective connected curve $C$ and a $k$-rational point $p \in C$. 

(1) A morphism $\pi\colon Y \to C$ of scheme is a \emph{double-point degeneration} over $p\in C$ if $\pi^{-1} (p)$ can be written as $\pi^{-1} (p) = A \cup B,$ where $A$ and $B$ are smooth closed subschemes of $Y$ of codimension $1$ intersecting transversally. The intersection $D = A \cap B$ is called the \emph{double-point locus} of $\pi$ over $p \in C$. We allow $A, B,$ or $D$ to be empty. Let $N_{A/D}$ and $N_{B/D}$ denote the normal bundles of $D$ in $A$ and $B$, respectively. As in \cite[\S 0.2]{LP}, the projective bundles $ \mathbb{P} (O_D \oplus N_{A/D})$ and $\mathbb{P} (O_D \oplus N_{B/D})$ over $D$ are isomorphic. Either of these is denoted by $\mathbb{P} (\pi) \to D$.

(2) Let $X \in \Sch_k$ and let $pr_1, pr_2$ be the projections from $X \times C$ to $X$ and $C$, respectively. Let $Y \in \Sm_k$ be of pure dimension, and let $g\colon Y \to X \times C$ be a projective morphism such that $\pi = pr_2 \circ g \colon Y \to C$ is a double-point degeneration over $p\in C$. For each regular value  $\zeta \in C(k)$ of $\pi$, the triple $(g, p, \zeta)$ is called a \emph{double-point cobordism} with the degenerate fiber over $p \in C$ and the smooth fiber over $\zeta$. The associated \emph{double-point relation} over $X$ is 
given by 
\[
\partial_C (g, p, \zeta)\colon = [Y_{\zeta} \to X] - [A \to X] - [B \to X] + [\mathbb{P} (\pi) \to X]  = 0 \ {\rm in} \ \  \mathcal{M}_* (X)^+,
\]
where $Y_{\zeta} \colon= \pi^{-1} (\zeta)$. 

(3) Let $\mathcal{R}_* ^{\rm rat} (X) \subset \mathcal{M}_* (X)^+$ be the subgroup generated by all double-point relations over $X$ over the pair $(C, p) = (\mathbb{P}^1, 0)$. This is the group of \emph{rational double-point relations}. This group was denoted by $\mathcal{R}_* (X)$ in \cite{LP}.

(4) Let $\mathcal{R}_* ^{\rm alg} (X) \subset \mathcal{M}_* (X)^+ $ be the subgroup generated by all double-point relations over $X$ over all pairs $(C, p)$ of smooth projective connected curve $C$ and a point $p \in C(k)$. This is the group of \emph{algebraic double-point relations}. 
\end{definition}


\begin{definition}[Levine-Pandharipande]\label{def:algebraic dp-cobordism}Let $X \in \Sch_k$.

(1) The \emph{(rational) double-point cobordism theory} $\omega_* (X)$ is the quotient 
\[
\omega_*(X)= \mathcal{M}_* (X)^+/ \mathcal{R}_* ^{\rm rat} (X). 
\]

(2) The \emph{algebraic double-point cobordism theory} $\omega_* ^{\rm alg} (X)$ is the quotient 
 \[
 \omega_* ^{\rm alg} (X) = \mathcal{M}_* (X)^+ / \mathcal{R}_* ^{\rm alg} (X). 
 \]  
\end{definition}

\subsection{Basic structures}
Some of the following basic properties of $\omega^{\rm alg}_*$ follow
essentially from the definition and some analogous
constructions in \cite[\S 3.1]{LP}.

\begin{proposition}\label{prop:Basic-small}
The functor $X \mapsto \omega_* ^{\rm alg} (X)$ on $\Sch_k$ has the following 
structures. 

\emph{(1) Projective push-forward:} For a projective morphism
$g\colon X \to X'$, we have $g_* \colon \omega_* ^{\rm alg} (X) \to \omega_* ^{\rm alg} (X')$ given by $g_* ([f\colon Y \to X]) = [g \circ f \colon Y \to X']$. This satisfies $(g_1 \circ g_2)_* = {g_1}_* \circ {g_2}_*$ when $g_1$ and $g_2$ are both projective.

\emph{(2) Smooth pull-back:} For a smooth quasi-projective morphism
$g\colon X' \to X$ of relative dimension $d$, we have $g^* \colon \omega_* ^{\rm alg} (X) \to \omega_{* + d} ^{\rm alg} (X')$ given by $g^* ( [ f\colon Y \to X]) = [pr_2 \colon Y \times_X X' \to X']$. This satisfies $(g_1 \circ g_2)^* = g_2 ^* \circ g_1 ^*$ when $g_1$ and $g_2$ are both smooth.

\emph{(3) External product:} We have $\times \colon \omega_* ^{\rm alg} (X) \times  \omega_* ^{\rm alg} (X') \to \omega_* ^{\rm alg} (X \times X')$ given by $[f\colon Y \to X] \times [f' \colon Y' \to X'] = [f \times f' \colon Y \times Y' \to X \times X']$.

\emph{(4) Unit:} The class $1_{\Spec (k)} \in \omega_0 ^{\rm alg}(k)$ is the unit for the external product on $\omega_* ^{\rm alg}$. 

\emph{(5) Chern classes:} For every line bundle $L$ on $X$, there is a Chern class operation $\tilde{c}_1(L) \colon \omega^{\rm alg}_*(X) \to \omega^{\rm alg}_{*-1}(X)$ which is compatible with smooth pull-back and projective push-forward.
\end{proposition}

\begin{proof} (1) Given a projective morphism $g\colon X \to X'$, we already have $g_* \colon \mathcal{M}_* (X)^+ \to \mathcal{M}_* (X')^+$. It remains to show that $g_*$ sends the algebraic double-point relations $\mathcal{R}_* ^{\rm alg} (X)$ into  $\mathcal{R}_* ^{\rm alg} (X')$. Indeed, given an algebraic double-point cobordism $(h, p, \zeta)$ over $X$, where $h\colon Y \to X \times C$ with a smooth projective connected curve $C$, we get an algebraic double-point cobordism $( (g \times {\rm Id}_C ) \circ h, p, \zeta)$ over $X'$, where $(g \times {\rm Id}_C) \circ h\colon Y \to X' \times C$. We immediately note that $g_* (\partial_C (h, p, \zeta)) = \partial_C ((g \times {\rm Id}_C) \circ h, p, \zeta)$. This proves (1).

(2) Given a smooth and quasi-projective morphism $g \colon X' \to X$, we have $g^* \colon \mathcal{M}_* (X) ^+ \to \mathcal{M}_* (X')^+$. It remains to show that $g^*$ sends $\mathcal{R}_* ^{\rm alg} (X)$ into $\mathcal{R}_* ^{\rm alg} (X')$. This follows by observing that given an algebraic double-point cobordism $(h, p, \zeta)$ over $X$, the pull-back $(g^* h, p, \zeta)$, given by the second projection of the fiber product $Y' \colon = Y \times _{X \times C} (X' \times C) \to X' \times C$, is an algebraic double-point cobordism over $X'$.

(3) The map $\times \colon \mathcal{M}_* (X)^+ \times \mathcal{M}_* (X')^+ \to \mathcal{M}_* (X \times X')^+$  is defined on the level of cobordism cycles. For an algebraic double-point cobordism $(h, p, \zeta)$ over $X$ as before, for each $[f\colon Y' \to X'] \in \mathcal{M}_* (X')^+$, we get an induced algebraic double-point cobordism $ (h \times f, p, \zeta)$, where $h \times f \colon Y \times Y' \to X \times X' \times C$. Similarly, interchanging the role of $X$ and $X'$, we see that $\times$ descends onto the level of $\omega_* ^{\rm alg} (-)$. This proves (3). Part (4) is immediate.

(5) The construction of the first Chern class operation $\tilde{c}_1(L)$ on $\omega^{\rm alg}_*(X)$ follows the same arguments as for $\omega_*(X)$ by first assuming that $L$ is globally generated and then deducing the general case, as in \cite[\S 4 and \S 9]{LP}. We omit the details.
\end{proof}

\section{The basic exact sequence}\label{section:Main-sequence}

Let $X \in \Sch_k$. By \cite[Theorem 1]{LP}, the natural map $\omega_*(X) \to \Omega_*(X)$ is an isomorphism. We often use this identification implicitly. Let $(C, t_1, t_2)$ denote a smooth projective connected curve $C$ with distinct points $t_1, t_2 \in C(k)$ with the inclusions $i_j \colon X \times\{ t_j \} \to X \times  C$. By the existence of the l.c.i. pull-backs on $\Omega_*$ in \cite[\S 6.5]{LM}, we have maps $i^*_1 , i^*_2 \colon \omega_*(X \times C) \to \omega_*(X)$. By definition, we also have a natural surjection $\Psi_X \colon\omega_*(X) \to \omega^{\rm alg}_*(X)$. The main theorem of the section is:

\begin{theorem}\label{thm:FES}Let $X \in \Sch_k$. The sequence
\[
{\underset{(C,t_1, t_2)}\bigoplus} \omega_*(X \times C) \xrightarrow{i^*_1 - i^*_2} \omega_*(X) \xrightarrow{\Psi_X} \omega^{\rm alg}_*(X) \to 0,
\]
where $(C, t_1, t_2)$ runs over the equivalence classes of all triples consisting of a smooth projective connected curve $C$ and two distinct points $t_1, t_2 \in C(k)$, is exact.
\end{theorem}

We begin with some remarks on cobordism cycles associated to strict normal crossing divisors on smooth schemes.

\subsection{Remarks on divisor classes}\label{background for Omega adpc} 
Recall from \cite[\S 3.1]{LM} that given a strict normal crossing divisor $E$ on $Y \in \Sm_k$ with the support $\iota\colon |E| \to Y$, there is a class $[E \to |E|] \in \Omega_* (|E|)$ that satisfies $\iota_* ([E \to |E|]) = [Y \to Y, O_Y (E)] = \tilde{c}_1 (O_Y (E))(1_Y)$. 
Since we have a natural surjection $\Omega_*  \to \Omega_* ^{\rm alg}$, the class $[E \to |E|]$ makes sense also in $\Omega_* ^{\rm alg} (|E|)$. 

The construction $[E \to |E|] \in \Omega_* (|E|)$ uses the formal group law $F$ for $\Omega_*(k)$. We look at only the following case from \cite[\S 3.1]{LM}. The special case we need is when $E= E_1 + E_2$, where $E_1$ and $ E_2$ are transversal smooth divisors on $Y \in \Sm_k$.  Let $\iota_D \colon D = E_1 \cap E_2 \to Y$ be the inclusion. We let $O_D (E_i)\colon= \iota_D ^* \left(O_Y (E_i)\right)$. 
The class $[E \to Y] \in \Omega_* (Y)$ is defined as
\[
 [E\to Y] \colon=  [E_1 \to Y] + [E_2 \to Y]  +{\iota_D }_* \left( F^{1,1} ( \tilde{c}_1 (O_D (E_1)) , \tilde{c}_1 (O_D (E_2)))(1_D) \right),
\]
where $F^{1,1}(u,v) = \sum_{i, j \geq 1} a_{i,j} u^{i-1} v^{j-1}\in \Omega_*  (k)[[u,v]]$ and $a_{i,j}\in \Omega_{i+j -1} (k)$ are the coefficients of the formal group law.

In addition, suppose that $O_D (E)\colon= \iota_D ^* \left(O_Y (E)\right)$ is trivial, \emph{i.e.}, $O_D (E_1) \simeq O_D (E_2)^{-1}$ on $D$. Let $\mathbb{P}_D \to D$ be the $\mathbb{P}^1$-bundle $\mathbb{P} (O_D \oplus O_D (E_1))$. By \cite[Lemma 3.3]{LP}, we have
\begin{equation}\label{pre-adpc equation}
F^{1,1} (\tilde{c}_1 (O_D (E_1)), \tilde{c}_1 (O_D (E_2))) (1_D) = - [\mathbb{P}_D \to D] \in \Omega_* (D).
\end{equation}
Hence, we have the following equation in $\Omega_* (Y)$ (and hence in 
$\Omega_* ^{\rm alg} (Y)$):
\begin{equation}\label{adpc equation}
[E \to Y] - [E_1 \to Y] - [E_2 \to Y] + [ \mathbb{P}_D \to Y] = 0.
\end{equation}

\subsection{Proof of Theorem \ref{thm:FES}}  Consider the commutative diagram with the top exact row:
\begin{equation}\label{eqn:FES*1}
\xymatrix{ 0 \to \mathcal{R}^{\rm alg} _* (X) \ar[r] \ar[rd] ^{\theta} & \mathcal{M} _* (X)^+ \ar@{->>}[d] \ar[r] & \omega_* ^{\rm alg} (X) \ar@{=}[d] \ar[r] & 0 \\
 \underset{   (C, t_1, t_2)}{\bigoplus}\omega_* (X \times C) \ar[r] _{ \ \ \ \ \ \theta'} & \omega_* (X) \ar[r]_{\Psi_X} & \omega_* ^{\rm alg} (X) \ar[r] & 0, }
\end{equation} 
where $\theta$ is the composition of the two arrows and $\theta'$ is the sum of the maps $i^*_1-i^*_2$. We want to prove that the bottom row is exact. It is apparent that $\ker ( \Psi_X ) = {\rm Im} (\theta)$, thus it suffices to prove that ${\rm Im} (\theta) = {\rm Im} (\theta ')$.

We prove ${\rm Im} (\theta) \subseteq {\rm Im} (\theta')$ first. Let $(g, p, \zeta)$ be a double-point cobordism as in Definition \ref{double point degeneration}, \emph{i.e.}, a projective $g \colon Y \to X \times C$, two points $p, \zeta \in C(k)$ such that for $\pi = pr_2 \circ g$ we have $\pi^{-1} (p) = A \cup B$. 
Set $\gamma\colon= [g\colon Y \to X \times C]\in \omega_* (X \times C)$. 


Let $i_p \colon X\times\{p\} \to X \times C$ be the inclusion and let
$X_p\colon=X \times \{ p\}$. Since the divisor $E\colon=g^{*}(X_p) = A + B$ is strict normal crossing, we have $\gamma \in \Omega_* (X \times C)_{X_p}$ (see Definitions \ref{defn:cobord cycle with D}, \ref{defn:cobord cycle with D 1}, \ref{defn:cobord cycle with D 2}).
By Theorem \ref{thm:CML}, Definition \ref{defn:Intersection*} and \cite[Lemma 6.5.6]{LM}, we see that $i^*_p(\gamma) = g'_*([E \to |E|]) \in \omega_*(X_p)$, where $g'=g|_{|E|}\colon|E| \to X_p$. Consider now the commutative diagram:
\[
\xymatrix{
|E| \ar[r]^{\iota_E} \ar[d]_{g'} & Y \ar[d]^{g} \ar[dr]^{\pi'} & \\
X_p \ar[r]^{i_p} & X \times C \ar[r]^>>>>>{pr_1} & X.}
\]

Note that $pr_1 \circ i_p = {\rm Id}_X$ via $X \simeq X_p$ and $\pi'$ is projective. Thus, $i^*_p (\gamma)= g_*' ([E \to |E|]) = {pr_1} _*  {i_p}_* g'_*\left([E \to |E|]\right) = \pi'_* {\iota_{E}}_*\left([E \to |E|]\right) = \pi'_*([E \to Y]) =^{\dagger} [A \to X] + [B \to X ] - [ \mathbb{P}(\pi) \to X]$ in $\omega_*(X)$, where $\dagger$ follows from \eqref{adpc equation}. Since $Y_\zeta$ is smooth, $i^*_\zeta(\alpha) = [Y_\zeta \to X]$. Hence, we get $\theta (\partial_C(g, p, \zeta)) =[Y_{\zeta} \to X] - [A \to X] - [B \to X ] + [ \mathbb{P}(\pi) \to X] =  -(i_p ^* - i_{\zeta}^*)(\gamma)$. That is, ${\rm Im} (\theta ) \subseteq {\rm Im} (\theta')$.

To prove the reverse inclusion ${\rm Im} (\theta) \supseteq {\rm Im} (\theta')$, we consider two cases.

\emph{Case 1:} First assume that $X$ is smooth. For $(C, t_1, t_2)$ as before, let $\gamma\colon=[g\colon Y \to X \times C]$ be a cobordism cycle. Since $X$ is smooth, by the transversality \cite[Proposition 3.3.1]{LM}, we may assume that $g$ is transverse to both $i_1$ and $i_2$. The composition $Y \to X \times C \to C$ now has smooth fibres over $t_1, t_2$ so that we have $-(i^*_{1} - i^*_{2}) (\gamma) = \theta \left(\partial_C(g, t_1, t_2)\right)$. So, if $X$ is smooth, then ${\rm Im} (\theta) \supseteq {\rm Im} (\theta')$ holds.

\emph{Case 2:} Suppose $X$ is any scheme. We prove by induction on $\dim X$. Note that every cobordism cycle is a formal sum of cobordism cycles of the form $[f \colon Y \to X]$ where $Y$ is smooth and integral, and such $f$ factors uniquely through an irreducible component of $X_{\rm red}$. Thus, we may reduce to the case when $X$ is integral.

If $\dim X = 0$, then $X$ is smooth so that the statement holds by \emph{Case 1}. Suppose $\dim X >0$, and assume the statement holds for all lower dimensional schemes in $\Sch_k$.

Let $\iota\colon Z \hookrightarrow X$ be the singular locus and let $U\colon= X \backslash Z$ be the open complement. Using Hironaka's resolution of singularities, we can find a projective morphism $\pi\colon \widetilde{X} \to X$ that is an isomorphism over $U$ such that the inverse image of $Z$ is a strict normal crossing divisor. Let $[g\colon Y \to X \times C]\in \omega_* (X \times C)$, and let $t_1, t_2 \in C(k)$ be two distinct points. Consider the diagram:
\[
\xymatrix{ E \ar[d]  \ar@{^{(}->}[r]  & \widetilde{Y} \ar[d] ^\mu \ar[r] ^{\widetilde{g}} & \widetilde{X} \times C  \ar[d] ^{\pi_C} & U \times C \ar@{_{(}->}[l] \ar@{=}[d] \\
W \ar@{=}[rd] \ar@{^{(}->}[r]^j & Y \ar[r] ^g  \ar@{-->}[ru]|-{f} & X \times C & U \times C\ar@{_{(}->}[l] \\
& W \ar@{^{(}->}[u]_j \ar[r]^{g'} & Z \times C \ar@{^{(}->}[u] _{\iota_C}, &}
\]
where $W \colon= g^{-1} (Z\times C)$, $g'$ is the restriction of $g$ on $W$, $f$ is the rational map $\pi_C ^{-1} \circ g$, and $\mu$ is a sequence of blow-ups of the indeterminacy of $f$, which is an isomorphism on the complement of $W$ such that the exceptional divisor $E$ is a strict normal crossing divisor, and such that there is a morphism $\widetilde{g}$ making the diagram commute. Moreover, the upper-right and the lower squares are Cartesian.

Let $\alpha \colon= [ g\colon Y \to X \times C]\in \omega_* (X \times C), \widetilde{\alpha} \colon=[ \widetilde{g}\colon \widetilde{Y} \to \widetilde{X} \times C] \in \omega_* (\widetilde{X}\times C)$, and $\beta\colon= [\mu\colon \widetilde{Y} \to Y]\in \omega_* (Y)$. Recall that for $V\in \Sm_k$, we write $1_V= [ {\rm Id} \colon V \to V] \in \omega_* (V)$. Then as cobordism cycle classes, we have
\begin{equation}
\alpha = g_* (1_Y), \ \ \widetilde{\alpha} = \widetilde{g}_* (1_{\widetilde{Y}}), \ \ {\pi_C}_* (\widetilde{\alpha}) = g_* \mu_* (1_{\widetilde{Y}}) = g_* (\beta).
\end{equation}
Thus, $\alpha -  {\pi_C}_* (\widetilde{\alpha} )= g_* (1_Y - \beta).$ 
The blow-up formula \cite[Proposition 3.2.4]{LM} implies that there is a cobordism cycle $\eta \in \omega_* (W)$ such that $1_Y- \beta = j_* (\eta)$. We thus have
\begin{equation}
\alpha - {\pi_C}_* (\widetilde{\alpha}) = g_* (1_Y-\beta) =  g_* j_* (\eta) ={\iota_C} _* g'_* (\eta).
\end{equation} 
In particular, we have $i_j ^* (\alpha) - i_j ^* ({ \pi_C} _* (\widetilde{\alpha})) = i_j ^* ({\iota_C} _* g'_* (\eta))$ for $j=1,2$ so that 
\begin{equation}\label{eqn:middle summary}
\theta' (\alpha) - \theta' ({\pi_C }_* (\widetilde{\alpha})) = \theta' ({\iota_C} _* g'_* (\eta)).
\end{equation}

On the other hand, in the Cartesian diagrams below whose rows are regular embeddings,
\[
\xymatrix@C.9pc{
\widetilde{X} \times \{t_j\} \ar[r] \ar[d]^{\pi} & \widetilde{X} \times C \ar[d]^{\pi_C} &
Z \times \{t_j\} \ar[r] \ar[d]^{\iota} & Z \times C \ar[d]^{\iota_C} \\
X \times \{t_j\} \ar[r] & X \times C, & X \times \{t_j\} \ar[r] & X \times C,} 
\] 
we can use \cite[Proposition 6.5.4]{LM} 
to deduce that $\theta' ({\pi_C }_* (\widetilde{\alpha})) = \pi_*\left(\theta'(\widetilde{\alpha})\right)$ and $\theta' ({\iota_C} _* g'_* (\eta)) = {\iota}_*\left(\theta'(g'_* (\eta)) \right)$. (N.B. : The Tor-independence assumption in \cite[Proposition 6.5.4]{LM} is only to guarantee that pull-backs of regular embeddings are regular embeddings. In our case, the rows are regular embeddings, thus, the proposition applies here.) 
\\
Applying this to \eqref{eqn:middle summary}, we conclude that
\begin{equation}\label{eqn:Tor-ind1}
\theta' (\alpha) = \pi_*\left(\theta'(\widetilde{\alpha})\right) + {\iota}_*\left(\theta'(g'_* (\eta)) \right).
\end{equation}

By the \emph{Case 1} applied to $\widetilde{X}$, we have that $\theta' (\widetilde{\alpha}) \in  \theta (\mathcal{R}_* ^{\rm alg} (\widetilde{X}))$ so that ${\pi}_* (\theta' (\widetilde{\alpha})) \in {\pi}_*( \theta (\mathcal{R}_* ^{\rm alg}(\widetilde{X})))\subset \theta (\mathcal{R}_* ^{\rm alg}(X))$ by Proposition \ref{prop:Basic-small}. Thus, to show $\theta' (\alpha) \in \theta (\mathcal{R}_* ^{\rm alg} (X))$, it is enough to prove that $\theta' ( g'_* (\eta)) \in \theta (\mathcal{R} _* ^{\rm alg} (Z))$. But this holds by the induction hypothesis since $\dim Z < \dim X$.
Hence, we have shown that ${\rm Im} (\theta) \supseteq {\rm Im} (\theta')$ for $X$. This finishes the proof of the theorem.\hfill $\Box$



\section{Equivalence of $\Omega_* ^{\rm alg}$ and $\omega_* ^{\rm alg}$ and consequences}\label{section:comparison}
The purpose of this section is to prove Theorem \ref{thm:comparison}, and to establish some fundamental properties of our cobordism theory.

\subsection{The comparison theorem}

First, we state an analogue of \cite[Lemma 3.2]{LP} for algebraic equivalence:

\begin{lemma}\label{equivalent divisors Omega} Let $Y \in \Sm_k$ and let $E, F$ be strict normal crossing divisors on $Y$ that are algebraically equivalent. Then $[E \to Y] = [F \to Y]$ in $\Omega_* ^{\rm alg} (Y)$.
\end{lemma}

\begin{proof}By \cite[Proposition 3.1.9]{LM}, we have $[E\to Y] = [Y \to Y, O_Y (E)]$ and $[F \to Y] = [Y\to Y, O_Y (F)]$ in $\Omega_* (Y)$. Via the natural map $\Omega_* (Y) \to \Omega_* ^{\rm alg} (Y)$, these equalities still hold in $\Omega_* ^{\rm alg} (Y)$. It follows from the relation (Equiv) of Definition \ref{precobordism definition} and Lemma \ref{lem:trivial**} that $[Y \to Y, O_Y (E)]= [Y \to Y, O_Y (F)]$ in $\Omega_* ^{\rm alg} (Y)$. Hence $[E \to Y] = [F\to Y]$ in $\Omega_* ^{\rm alg} (Y)$.
\end{proof}

\begin{theorem}\label{thm:comparison} For $X \in \Sch_k$, there is a canonical isomorphism $\Omega_{*} ^{\rm alg} (X) \simeq \omega_* ^{\rm alg} (X)$. 
\end{theorem}

\begin{proof}
We first define a map $\vartheta^{\rm alg}_X\colon \omega^{\rm alg}_*(X) \to \Omega^{\rm alg}_*(X)$. We let $\vartheta^{\rm alg}_X \colon \mathcal{M} _* (X)^+ \to \Omega_* ^{\rm alg} (X)$ be given by $\vartheta^{\rm alg}_X ([f\colon Y \to X]_{\omega^{\rm alg}}) \colon= [f \colon Y \to X]_{\Omega^{\rm alg}}$. We need to show that $\vartheta^{\rm alg}_X$ kills the algebraic double-point relations.

So let $(g, p, \zeta)$ be an algebraic double-point cobordism given by a projective $g\colon Y \to X \times C$, where $C$ is a smooth projective curve. It is enough to show that $\partial_C(g, p, \zeta)$ vanishes in $\Omega_* ^{\rm alg} (X)$. Let $f\colon= pr_1 \circ g$ and $\pi \colon= pr_2 \circ g$. Since 
\[
\partial_C(g, p, \zeta) =f_*\left([Y_\zeta \to Y] - [A \to Y] - [B \to Y] + [ \mathbb{P}(\pi) \to Y]\right),
\]
it suffices to show that the relation
\begin{equation}\label{eqn:comp2}
[Y_\zeta \to Y] - [A \to Y] - [B \to Y] + [\mathbb{P}(\pi) \to Y] = 0
\end{equation}
holds in $\Omega_* ^{\rm alg} (Y)$.

We apply the equation \eqref{adpc equation} to the divisor $E\colon= A + B$ on $Y$ to obtain 
\begin{equation}\label{adpc Omega eqn1}
[E \to Y] - [A \to Y] - [B \to Y] + [\mathbb{P}(\pi) \to Y] = 0 \in \Omega_* ^{\rm alg} (Y).
\end{equation}
On the other hand, the divisor $E$ is algebraically equivalent to the divisor $Y_{\zeta}$ and hence by Lemma \ref{equivalent divisors Omega}, we also have the equality $[E \to Y] = [Y_{\zeta} \to Y] \in \Omega_* ^{\rm alg} (Y)$. Combining this with \eqref{adpc Omega eqn1}, we obtain \eqref{eqn:comp2}. Thus, the map $\vartheta_X ^{\rm alg} \colon \mathcal{M}_* (X)^+\to \Omega_* ^{\rm alg}(X)$ descends to give $\vartheta^{\rm alg}_X \colon \omega^{\rm alg}_*(X) \to \Omega^{\rm alg}_*(X)$.

To define the inverse $\tau^{\rm alg}_X\colon \Omega_{*} ^{\rm alg} (X) \to \omega_* ^{\rm alg} (X)$ of $\vartheta^{\rm alg}_X$, we consider the diagram 
\begin{equation}\label{eqn:comp1}
\xymatrix{
& \Omega_*(X) \ar[r]^{\Phi_X} \ar[d]_{\tau_X} & 
\Omega_{*} ^{\rm alg}(X) \ar@{-->}[d] \ar[r] & 0 \\
{\underset{(C,t_1, t_2)}\bigoplus} 
\omega_*(X \times C) \ar[r]_>>>>{i^*_1 - i^*_2} & \omega_*(X) \ar[r]_{\Psi_X} &
\omega^{\rm alg}_*(X) \ar[r] & 0,}
\end{equation}
where the bottom row is exact by Theorem \ref{thm:FES} and the isomorphism $\tau_X$ is from \cite[\S 11.1]{LP}. 

Let $\theta'$ be the sum of the maps $i_1 ^* - i_2 ^*$. We need to show that $\tau_X (\ker (\Phi_X)) \subseteq {\rm Im} (\theta')$ in order to define $\tau_X ^{\rm alg}$. By Proposition \ref{prop:Alg-st}, $\ker (\Phi_X)$ is generated by cobordism cycles $\alpha$ of the form $[f\colon Y \to X, L_1]-[f\colon Y \to X, L_2]$ such that $L_1 \sim L_2$. So, it is enough to show that $\tau_X (\alpha) \in {\rm Im} (\theta')$ for such $\alpha$. We can write $\alpha = f_* \left(\tilde{c}_1(L_1)(1_Y)- \tilde{c}_1(L_2)(1_Y)\right)$. Applying \cite[Theorem 6.5.12]{LM} (the Tor-independence assumptions hold by Lemma \ref{lem:Tor-ind} below) to the 
Cartesian square
$$\begin{CD}
Y \times \{ t_j \} @>{i_j}>> Y \times C \\
@V{f}VV @V{f_C}VV\\
X \times \{ t_j \} @>{i_j}>> X \times C,
\end{CD}$$ we deduce that $\theta'$ respects projective push-forwards. Since $\tau_X$ also respects projective push-forwards, we may replace $X$ by $Y$, $f$ by ${\rm Id}_X$, and $\alpha$ by $\tilde{c}_1(L_1)(1_Y)- \tilde{c}_1(L_2)(1_Y)$.


Since $L_1 \sim L_2$, there exists a smooth projective curve $C$, two distinct points $t_1, t_2 \in C(k)$ and a line bundle $\mathcal{L}$ on $X \times C$ such that $\mathcal{L}|_{X \times \{t_j\}}  \simeq L_i$ for $j = 1, 2$. We can then write 
\[
[{\rm Id}_X\colon X \to X, L_j]  = \tilde{c}_1(L_j)(1_X)  = ( \tilde{c}_1\left(i^*_j(\mathcal{L})\right) \circ  i^*_j )\left(1 _{X \times C}\right) = ( i^*_j \circ \tilde{c}_1(\mathcal{L}))\left(1_{X \times C}\right)
\]
where the last equality follows from \cite[Lemma 7.4.1 (2)]{LM}.

In particular, we see that $\alpha = i^*_1(\widetilde{\alpha}) - i^*_2(\widetilde{\alpha})$, where $\widetilde{\alpha} = [{\rm Id}\colon X \times C \to X \times C, \mathcal{L}]$. That is, $\tau_X(\alpha) = \theta'(\widetilde{\alpha})$. This shows that $\tau_X(\ker (\Phi_X))
\subseteq {\rm Im}(\theta')$ and it defines $\tau^{\rm alg}_X$ such that the above diagram commutes. 

Both $\omega_* ^{\rm alg} (X)$ and $\Omega_* ^{\rm alg} (X)$ are generated by cobordism cycles of the form $[f\colon Y \to X]$, and for those cycles, $\tau^{\rm alg}$ and $\vartheta^{\rm alg}_X$ are inverses of each other. This proves the theorem. 
\end{proof}

We used the following basic lemma in the proof of the above. We shall use it again to prove Theorem \ref{thm:localization**}.

\begin{lemma}\label{lem:Tor-ind}
Let $T$ be a smooth scheme and let $W \subset T$ be a smooth closed subscheme. Then for any morphism $f \colon V' \to V$ in $\Sch_k$, the schemes $V' \times T$ and $V \times W$ are Tor-independent over $V \times T$.

In particular, if $C$ is a smooth curve and $\{t\}$ is a point in $C(k)$, then for any morphism $f \colon Y \to X$ in $\Sch_k$, the schemes $Y \times C$ and $X \times \{ t \}$ are Tor-independent over $X \times C$.
\end{lemma}
\begin{proof}
Since the first assertion is local on $V$ and $T$, by shrinking them to small enough affine open subschemes if necessary, we may assume that both are affine such that $W \subseteq V$ is a complete intersection subscheme. In particular, there is a finite resolution $\mathcal{F}_{\bullet} \to \mathcal{O}_W$ by free $\mathcal{O}_T$-modules of finite rank. This in turn shows that $\mathcal{F}_{\bullet} \otimes_k \mathcal{O}_V \to \mathcal{O}_{V\times W}$ is a finite free resolution of $\mathcal{O}_{V\times W}$ as $\mathcal{O}_{V \times T}$-module.

Since $\left(\mathcal{F}_{\bullet} \otimes_k \mathcal{O}_V\right) \otimes_{\mathcal{O}_{V \times T}} \mathcal{O}_{V' \times T} \simeq \mathcal{F}_{\bullet} \otimes_k \mathcal{O}_{V'}$,  we have 
\[
\Tor^{\mathcal{O}_{V \times T}}_i \left(\mathcal{O}_{V \times W}, \mathcal{O}_{V' \times T}\right) = \mathcal{H}_i\left(\mathcal{F}_{\bullet} \otimes_k \mathcal{O}_{V'}\right) = 0
\]
for $i > 0$. The second assertion is a special case of the first. \end{proof}

As an immediate consequence of Theorems \ref{thm:FES} and \ref{thm:comparison},
we obtain:

\begin{theorem}\label{thm:FES-ST}For $X \in \Sch_k$, there is an exact sequence:
\[
{\underset{(C,t_1, t_2)}\bigoplus} \Omega_*(X \times C) \xrightarrow{i^*_1 - i^*_2} \Omega_*(X) \xrightarrow{\Phi_X} \Omega^{\rm alg}_*(X) \to 0,
\]
where $(C, t_1, t_2)$ runs over the equivalence classes of triples consisting of a smooth projective connected curve $C$ and two distinct points $t_1, t_2 \in C(k)$, and $i_j ^*$ is the pull-back via the closed immersion $i_j: X \times \{ t_j \} \to X \times C$ for $j=1, 2.$
\end{theorem}

\subsection{Fundamental properties of $\Omega^{\rm alg}_*$}
\label{subsection :Fundamental-STC}
We now prove some important properties of our cobordism theory. 

\begin{theorem}[Localization sequence]\label{localization for sim-cobordism}\label{thm:localization**}
Given a closed immersion $Y \inj X$ in $\Sch_k$ with complement $U$, there
is an exact sequence
\[
\Omega_* ^{\rm alg} (Y)\to  \Omega_* ^{\rm alg} (X) \to \Omega_* ^{\rm alg} (U) \to 0.
\]
\end{theorem}

\begin{proof}
Let $\iota\colon Y \to X$ and $j \colon U \to X$ be the inclusions. 
Consider the diagram 
\begin{equation}\label{eqn:localization square}
\xymatrix{
{\underset{(C,t_1, t_2)}\bigoplus} 
\Omega_*(Y \times C) \ar[r]^>>>>{i^*_1 - i^*_2} \ar[d]^{{\iota_C}_*} & \Omega_*(Y) 
\ar[r]^{\Phi_Y} \ar[d]^{\iota_*} & \Omega^{\rm alg}_*(Y) \ar[d]^{\iota_*} \ar[r] &
0 \\
{\underset{(C,t_1, t_2)}\bigoplus} 
\Omega_*(X \times C) \ar[r]^>>>>{i^*_1 - i^*_2} \ar[d]^{j_C^*} & \Omega_*(X) 
\ar[r]^{\Phi_X} \ar[d]^{j^*} & \Omega^{\rm alg}_*(X) \ar[d]^{j^*} \ar[r] &
0 \\
{\underset{(C,t_1, t_2)}\bigoplus} 
\Omega_*(U \times C) \ar[r]^>>>>{i^*_1 - i^*_2} \ar[d] & \Omega_*(U) 
\ar[r]^{\Phi_U} \ar[d] & \Omega^{\rm alg}_*(U) \ar[d] \ar[r] &
0 \\
0 & 0 & 0, & }
\end{equation}where $\iota_C\colon Y \times C \to X \times C$ and $j_C: U \times C \to X \times C$ are the induced inclusions. Here, the rows are exact by Theorem \ref{thm:FES-ST} and the first two columns are exact by \cite[Theorem 3.2.7]{LM}. This diagram commutes: the bottom left square commutes by the composition law of the pull-backs. The top and the bottom right squares commute by the naturality of the maps $\Phi_{(-)}$. For the top left square, consider the Cartesian square:
$$\begin{CD}
Y \times \{ t_j \} @>{1_Y \times i_j}>> Y \times C \\
@V{\iota}VV @V{\iota_C}VV\\
X \times \{ t_j \} @>{1_X \times i_j}>> X \times C,
\end{CD}$$ where the horizontal maps are l.c.i. morphisms and the vertical maps are closed immersions. By Lemma \ref{lem:Tor-ind}, $1_X \times i_j$ and $\iota_C$ are Tor-independent. Hence by \cite[Theorem 6.5.12]{LM}, we have 
$(1_X \times i_j )^*  \circ { \iota_C}_*= \iota_*  \circ (1_Y \times i_j )^*$, which implies that $\left( (1_X \times i_1 )^* - (1_X \times i_2)^* \right) \circ {\iota_C}_* = \iota_* \circ \left( (1_Y \times i_1 )^*- (1_Y \times i_2 )^* \right)$. This means that the top left square of \eqref{eqn:localization square} commutes. Thus, we have shown that the diagram \eqref{eqn:localization square} commutes. A simple diagram chase now shows that the third column is also exact. This proves the theorem.
\end{proof}

\begin{theorem}[$\mathbb{A}^1$-homotopy Invariance]\label{extended homotopy sim}
\label{thm:H-Invariance}
Let $X\in \Sch_k$ and let $p\colon V \to X$ be a torsor under a vector bundle over $X$ of rank $n$. Then the map $p^* \colon \Omega_* ^{\rm alg} (X) \to \Omega_{*+n} ^{\rm alg} (V)$ is an isomorphism. 
\end{theorem}
\begin{proof}
We consider the commutative diagram
\[
\xymatrix{
{\underset{(C,t_1, t_2)}\bigoplus} 
\Omega_*(X \times C) \ar[r]^>>>>>>{i^*_1 - i^*_2} \ar[d]^{p^*} & \Omega_*(X) 
\ar[r]^{\Phi_X} \ar[d]^{p^*} & \Omega^{\rm alg}_*(X) \ar[d]^{p^*} \ar[r] &
0 \\
{\underset{(C,t_1, t_2)}\bigoplus} 
\Omega_{*+n}(V \times C) \ar[r]^>>>>{i^*_1 - i^*_2} & \Omega_{*+n}(V) 
\ar[r]^{\Phi_V} & \Omega^{\rm alg}_{*+n}(V) \ar[r] & 0,}
\]
where the rows are exact by Theorem \ref{thm:FES-ST}. The first two vertical arrows are isomorphisms by \cite[Theorem 3.6.3]{LM}. Hence, so is the third arrow by diagram chasing.
\end{proof}

Using the projective bundle formula \cite[Theorem 3.5.4]{LM} for $\Omega_*$, the argument of the proof of Theorem \ref{thm:H-Invariance} can be repeated in verbatim with $V$ replaced by $\mathbb{P}(V)$ to prove the following projective bundle formula for our cobordism theory. 

\begin{theorem}[Projective bundle formula]\label{proj bund sim-cobordism}
\label{thm:PBF}
Let $X \in \Sch_k$ and let $E$ be a rank $n+1$ vector bundle on $X$. Then, we have $\bigoplus_{j=0} ^n \Omega_{* - n + j} ^{\rm alg} (X) \overset{\simeq}{\to} \Omega_* ^{\rm alg} (\mathbb{P}({E})).$
\end{theorem}

\section{$\Omega_{\rm alg}^*$ as an oriented cohomology theory}\label{section:OCT}
Recall from \cite[Definition 1.1.2]{LM} that an oriented cohomology theory $A^*$ on $\Sm_k$ is an additive contravariant functor to the
category of commutative graded rings with unit, such that $A^*$ has push-forward maps for projective morphisms and it satisfies the $\mathbb{A}^1$-homotopy invariance and projective bundle formula. Moreover, the push-forward and the pull-back maps commute in a Cartesian diagram of transverse morphisms.

On the bigger category $\Sch_k$, there is a notion of an oriented Borel-Moore homology theory (see \cite[Definition 5.1.3]{LM}). This requires some similar axioms, but a nontrivial one is the existence of pull-backs for l.c.i. morphisms. This ensures that an oriented Borel-Moore homology theory on $\Sch_k$ restricted onto $\Sm_k$ gives an oriented cohomology theory.

Our goal in this section is to conclude that $\Omega^*_{\rm alg}$ is an oriented cohomology theory on $\Sm_k$ and $\Omega_* ^{\rm alg}$ is an oriented Borel-Moore homology theory on $\Sch_k$.

\subsection{Pull-back via l.c.i. morphisms}\label{sec:lci} By Definition \ref{basic definitions for cobordism cycles}, one can pull-back cobordism cycles via \emph{smooth} quasi-projective morphisms.
One further step is needed to turn $\Omega^{\rm alg}_*$ into an oriented Borel-Moore homology : to show that one can pull-back also via l.c.i. morphisms $f\colon X \to Y$ for $X, Y \in \Sch_k$. Recall that $f\colon X \to Y$ is an l.c.i. morphism if it factors as the composition $f= q \circ i\colon X\to P\to Y$, where $i$ is a regular embedding and $q$ is a smooth quasi-projective morphism. Since we have $q^*$ already, defining $i^*$ is the first technical issue to resolve. We shall demonstrate the existence of such pull-backs on $\Omega^{\rm alg}_*$ using Proposition \ref{prop:Alg-st} and the analogous construction for the algebraic cobordism in \cite[\S 5, 6]{LM}.

Recall from \cite[Definition 2.2.1]{Fulton} that a \emph{pseudo-divisor} $D$ on a scheme $X$ is a triple $D= (Z, \mathcal{L}, s)$, where $Z \subset X$ is a closed subset, $\mathcal{L}$ is an invertible sheaf on $X$, and $s$ is a section of $\mathcal{L}$ on $X$ such that the support of the zero scheme of $s$ is contained in $Z$. We call $Z$, \emph{the support of $D$} and write it as $|D|$. We call the zero scheme $\{s=0\}$, \emph{the divisor of $D$} and write it as $\Div(D)$. 

Given $X \in \Sch_k$ and a pseudo-divisor $D$ on $X$, Levine and Morel defined in \cite[\S 6.1.2]{LM} a graded group ${{\Omega}_*(X)}_D$ with a natural map $\theta_X\colon{{\Omega}_*(X)}_D \to \Omega_*(X)$, which is an isomorphism by \cite[Theorem 6.4.12]{LM}. Roughly speaking, this is the group on which the ``intersection product'' by the pseudo-divisor $D$ is well-defined so that we have a map $D(-)\colon {{\Omega}_*(X)}_D \to \Omega_{*-1}(|D|)$. (See \S~\ref{section:appendix} for the definitions of $\Omega_* (X)_D$ and $D(-)$.) This yields
\begin{equation}\label{eqn:Int-div}
i^*_D\colon \Omega_*(X) \stackrel{\theta^{-1}_X}{\underset{\simeq}\to} {{\Omega}_*(X)}_D \xrightarrow{D(-)} \Omega_{*-1}(|D|).
\end{equation} 

It follows from Proposition \ref{prop:Alg-st} and Lemma \ref{lem:Int-ps-div} that $i^*_D$ descends to
\begin{equation}\label{eqn:Int-div-st}
i^*_D\colon \Omega^{\rm alg}_*(X) \to \Omega^{\rm alg}_{*-1}(|D|).
\end{equation}

\subsubsection{Gysin map for regular embedding}Let $\iota_Z\colon Z \to X$ be a regular embedding of codimension $d$ in $\Sch_k$. We use \eqref{eqn:Int-div-st} and the technique of the deformation to the normal bundle to define the pull-back map $\iota^* _Z\colon \Omega_* ^{\rm alg} (X) \to \Omega_{*-d} ^{\rm alg} (Z)$, that we call \emph{the Gysin map} for the cobordism classes. 
Without going into the full construction of the deformation to the normal bundle, we recall here only the necessary summary from \cite[\S 6.5.2 (6.10)]{LM}:

\begin{proposition}\label{Gysin prop}Let $\iota_Z\colon Z \to X$ be a regular embedding in $\Sch_k$. Then, there exists a scheme $U \in \Sch_k$, a closed immersion $i_N \colon N \to U$ of codimension one, a surjective morphism $\mu\colon U \to X \times \mathbb{P}^1$, and its restriction $\mu_N \colon N \to Z \times 0$, that form the commutative diagram
\[
\xymatrix{ N \ar[rr] ^{i_N} \ar[d] _{\mu_N} & & U \ar[d] _{\mu} \\
Z \times 0 \ar[r]^{{\rm Id} \times 0} & Z \times \mathbb{P} ^1 \ar[r]^{\iota _Z\times {\rm Id}} & X \times \mathbb{P}^1}
\]
such that

\emph{(1)} $N$ is isomorphic to the normal vector bundle $N_{Z/X}$ of $Z$ in $X$ over $Z$ under the identification $Z = Z \times 0$, and

\emph{(2)} the restriction $\mu\colon U \backslash N \to X \times \left( \mathbb{P}^1 \backslash \{  0 \} \right)$ is an isomorphism of schemes.
\end{proposition}

We have the following analogue of \cite[Lemma 6.5.2]{LM}:
\begin{lemma}\label{push-pull}The composition $i_N ^* \circ {i_N}_* \colon \Omega_{*+1} ^{\rm alg} (N) \to \Omega_{*+1} ^{\rm alg} (U) \to \Omega_* ^{\rm alg} (N)$ is zero, where ${i_N}_*$ is the push-forward via the closed immersion $i_N$, and $i_N^*$ is the pull-back by the divisor $N$ defined in \eqref{eqn:Int-div-st}.
\end{lemma}

\begin{proof}Consider the commutative diagram
\[
\begin{CD}\Omega_{*+1} (N) @>{i_N ^* \circ {i_N}_*}>> \Omega_* (N) \\
@VVV @VVV\\
\Omega_{*+1} ^{\rm alg} (N) @>{i_N^* \circ {i_N}_*}>> \Omega_* ^{\rm alg} (N),\end{CD}
\]
where the vertical maps are the natural surjections. Since the top map on the algebraic cobordism is zero by \cite[Lemma 6.5.2]{LM}, the bottom map is also zero.\end{proof}

By Theorem \ref{localization for sim-cobordism}, we have the localization exact sequence 
\[
\Omega_{*+1} ^{\rm alg} (N) \overset{{i_N}_*}{\to} \Omega_{*+1} ^{\rm alg} (U) 
\overset{j^*}{\to} \Omega_{*+1} ^{\rm alg} (U \backslash N) \to 0,
\]
that gives an isomorphism 
\begin{equation}\label{conseq of local}
(j^*)^{-1} \colon \Omega_{*+1} ^{\rm alg} (U \backslash N) \to \frac{\Omega_{*+1}  ^{\rm alg} (U)}{ {i_N}_* ( \Omega_{*+1} ^{\rm alg} (N))}.\end{equation}
Combining \eqref{conseq of local} with Lemma \ref{push-pull}, we see that the composition 
\begin{equation}\label{conseq of local 2}
\alpha \colon \Omega_{*+1} ^{\rm alg} (U \backslash N) \overset{ (j^*)^{-1}}{\to} \frac{\Omega_{*+1}  ^{\rm alg} (U)}{ {i_N}_* ( \Omega_{*+1} ^{\rm alg} (N))} \overset{i_N^*}{\to} \Omega_* ^{\rm alg} (N)
\end{equation}
is well-defined.

\begin{definition}\label{defn:Gysin}For a regular embedding $\iota_Z \colon Z \to X$ of codimension $d$ in $\Sch_k$, the \emph{Gysin morphism} $\iota ^*_Z \colon \Omega_*^{\rm alg} (X) \to \Omega_{*-d} ^{\rm alg} (Z)$ is defined to be the composition
\[
\Omega^{\rm alg}_* (X) \overset{pr_1 ^*}{\to} \Omega^{\rm alg}_{*+1}(X \times (\mathbb{P}^1 \backslash \{0\} )) \underset{\simeq}{\overset{\mu^*}{\to} } \Omega^{\rm alg}_{*+1}(U \backslash N) 
\overset{\alpha}{\to} \Omega^{\rm alg}_{*}(N) 
\underset{\simeq}{\overset{(\mu_N ^*)^{-1}}{\to}} \Omega^{\rm alg}_{*-d}(Z),
\]
where $pr_1$ is the projection, $\mu$ is the isomorphism of Proposition \ref{Gysin prop}(2), $\alpha$ is the map in \eqref{conseq of local 2}, and $\mu_N\colon N \to Z$ is the normal bundle of Proposition \ref{Gysin prop}(1) so that $\mu_N ^*$ is an isomorphism by Theorem \ref{thm:H-Invariance}.
\end{definition}

We have the following basic properties for the Gysin maps on $\Omega_* ^{\rm alg}$ that can be easily deduced from \cite[Lemmas 6.5.6, 6.5.7, Theorem 6.5.8]{LM} combined with the surjectivity of $\Phi_X \colon \Omega_* (X) \to \Omega_* ^{\rm alg} (X)$, as in the proof of Lemma \ref{push-pull}. We skip the details:

\begin{proposition}\label{prop:Gysin}The Gysin maps on $\Omega_* ^{\rm alg}$ satisfy the following:

\emph{(1)} Let $\iota \colon Z \to X$ be a regular embedding of codimension one. Then, as operators $\Omega_* ^{\rm alg}(X) \to \Omega_{*-1} ^{\rm alg}(Z)$, the pull-back $Z(-)$ by the divisor $Z$ is identical to the Gysin pull-back $\iota^*$.

\emph{(2)} Let $\iota \colon Z \to X$ be a regular embedding, let $p\colon Y \to X$ be a smooth quasi-projective morphism, and let $s\colon Z \to Y$ be a section of $Y$ over $Z$. Then, $s^* \circ p^* = \iota^*$.

\emph{(3)} Let $\iota\colon Z \to Z'$ and $\iota' \colon Z ' \to X$ be regular embeddings. Then, $(\iota' \circ \iota)^* = \iota^* \circ {\iota'}^*$.

\end{proposition}

\subsubsection{Pull-back for l.c.i. morphisms} Let $f \colon X \to Y$ be an l.c.i. morphism in $\Sch_k$ with a factorization $f = p \circ i \colon X\to P \to Y$, where $p$ is smooth quasi-projective and $i$ is a regular embedding. We have $p^*$ by Definition \ref{basic definitions for cobordism cycles}, and we have the Gysin pull-back $i^*$ by Definition \ref{defn:Gysin}. So, one wishes to define $f^*$ by taking the composition $i^* \circ p^*$. To show that this definition is meaningful, one needs to know that if $p_1\circ i_1 = p_2 \circ i_2$ are two such factorizations, then $i_1 ^* \circ p_1 ^* = i_2 ^* \circ p_2 ^*$. However, this fact follows at once from such an equality on the level of algebraic cobordism, as shown in \cite[Lemma 6.5.9]{LM}, and from the surjection $\Phi _{ - }\colon \Omega_* (-) \to \Omega^{\rm alg}_* (-)$. Thus we have:

\begin{definition}\label{defn:lci pull-back}Let $f \colon X \to Y$ be an l.c.i. morphism that has a factorization $f = p \circ i\colon X \to P \to Y$, where $i$ is a regular embedding and $p$ is smooth quasi-projective. The pull-back $f^*$ on $\Omega_* ^{\rm alg} (Y)$ is defined to be $i^* \circ p^*$. 
\end{definition}

One has the following properties of the l.c.i. pull-backs on $\Omega_* ^{\rm alg}$ as proven for $\Omega_*$ in \cite[Theorems 6.5.11, 6.5.12, 6.5.13]{LM}. The proof follows immediately from \emph{ibids.} and we omit the arguments.

\begin{theorem}\label{thm:lci prop}The pull-backs on $\Omega_* ^{\rm alg}$ via l.c.i. morphisms have the following properties:

\emph{(1)} If $f_1 \colon X \to Y$, $f_2 \colon Y \to Z$ are l.c.i. morphisms in $\Sch_k$, then $(f_2 \circ f_1)^* = f_1 ^* \circ f_2 ^*$.

\emph{(2)} Suppose $f\colon X \to Z$ and $g\colon Y \to Z$ are Tor-independent morphisms in $\Sch_k$, where $f$ is l.c.i. and $g$ is projective. Then, for the Cartesian square 
\[
\begin{CD}
X \times_Z Y @>{pr_2}>> Y \\
@V{pr_1}VV @V{g}VV\\
X @>{f}>> Z,
\end{CD}
\] 
we have $f^* \circ g_* = {pr_1}_* \circ pr_2 ^*$.

\emph{(3)} Let $f_i \colon X_i \to Y_i$ for $i=1,2$ be two l.c.i. morphisms in $\Sch_k$. Then, for $\eta_i \in \Omega_* ^{\rm alg} (Y_i)$ with $i=1,2$, we have $ (f_1 \times f_2)^* (\eta_1 \times \eta_2) = f_1 ^* (\eta_1) \times f_2 ^* (\eta_2)$.
\end{theorem}

\begin{corollary}\label{pull-back for smooth varieties}Let $f\colon X \to Y$ be any morphism of smooth schemes. Then, there is a well-defined pull-back $f^*\colon \Omega^* _{\rm alg} (Y) \to \Omega^* _{\rm alg} (X)$. If $f\colon X \to Y$, $g\colon Y \to Z$ are any morphisms of smooth schemes, then $(g\circ f)^* = f^* \circ g^*$.
\end{corollary}

\begin{proof}Any morphism $f\colon X \to Y$ of smooth schemes is an l.c.i. morphism, with a factorization $f = pr_2 \circ \Gamma_f \colon X \to X \times Y \to Y$. The rest follows immediately.
\end{proof}

The main results proven in \S~\ref{subsection :Fundamental-STC} and 
\S~\ref{sec:lci} can now be summarized as follows:

\begin{theorem}\label{thm:universal}The theory $\Omega^* _{\rm alg}$ is an oriented cohomology theory on $\Sm_k$ that respects algebraic equivalence, and it is universal among such theories. In other words, for any oriented cohomology theory $A^*$ that respects algebraic equivalence, there exists a unique morphism of oriented cohomology theories $\theta \colon \Omega^* _{\rm alg} \to A^*$ on $\Sm_k$.

Similarly, the theory $\Omega_* ^{\rm alg}$ is an oriented Borel-Moore homology theory on $\Sch_k$ that respects algebraic equivalence, and it is universal among such theories. 
\end{theorem}

\section{Connections to algebraic cobordism, Chow groups and $K$-theory}\label{section:connection to cycle and K}
In this section, we study how our cobordism theory $\Omega_* ^{\rm alg}(X)$ is related with the Chow groups $\CH^{\rm alg} _* (X)$ modulo algebraic equivalence and the semi-topological $K$-groups $K_0 ^{\rm semi} (X)$ and $G_0 ^{\rm semi}(X)$.  
We shall also show that our cobordism theory agrees with the algebraic 
cobordism theory with finite coefficients.

\subsection{Connection with Chow groups and $K$-theory}

\begin{theorem}\label{equivalence with cycle}\label{thm:Cob-Chow}
For $X \in \Sch_k$, there is a natural map $\Omega^{\rm alg}_*(X) \to \CH^{\rm alg}_*(X)$ that induces an isomorphism $\Omega^{\rm alg}_*(X) \otimes_{\mathbb{L}_*} \mathbb{Z} \xrightarrow{\simeq} \CH^{\rm alg}_*(X).$
\end{theorem}
\begin{proof}
We consider the commutative diagram
\begin{equation}\label{eqn:Cob-G}
\xymatrix{
{\underset{(C,t_1, t_2)}\bigoplus} 
\Omega_*(X \times C) \ar[r]^>>>>>>{i^*_1 - i^*_2} \ar[d] & \Omega_*(X) 
\ar[r] \ar[d] & \Omega^{\rm alg}_*(X) \ar@{-->}[d] \ar[r] &
0 \\
{\underset{(C,t_1, t_2)}\bigoplus} 
\CH_*(X \times C) \ar[r]^>>>>{i^*_1 - i^*_2} & \CH_{*}(X) 
\ar[r] & \CH^{\rm alg}_{*}(X) \ar[r] & 0,}
\end{equation}
where the top row is exact by Theorem \ref{thm:FES-ST}. It follows from the definition of algebraic equivalence of algebraic cycles in \cite[Definition 10.3]{Fulton} and the proof of Lemma \ref{lem:trivial**} that the bottom row is also exact (see \cite[Example 10.3.2]{Fulton} when $k$ is algebraically closed). The existence of the first two vertical maps and their commutativity follow from the universal property of $\Omega_*$. This immediately yields a natural map $\Omega^{\rm alg}_*(X) \to \CH^{\rm alg}_*(X)$.

Moreover, the top row remains exact after applying the functor
$-\otimes_{\mathbb{L}_*}{\mathbb{Z}}$ and the first two vertical maps after tensoring are isomorphisms by \cite[Theorem 4.5.1]{LM}. Thus, the last vertical map after tensoring is also an isomorphism.
\end{proof}

\begin{remark}\label{univ Chow mod alg}By Theorems \ref{thm:universal}, \ref{equivalence with cycle}, and \cite[Theorem 1.2.2]{LM}, we see that $\CH^* _{\rm alg}$ is universal among oriented cohomology theories on $\Sm_k$ whose Chern class operations are additive, i.e., $\tilde{c}_1 (L_1 \otimes L_2) = \tilde{c}_1 (L_1) + \tilde{c}_1 (L_2)$ and respect algebraic equivalence.
\end{remark}

For $X \in \Sch_k$, let $K_0 (X)$ (resp. $G_0(X)$) be the Grothendieck group of coherent locally free sheaves (resp. coherent sheaves) on $X$. Recall from \cite[Definition 1.1]{FW} that the semi-topological $K$-group $K^{\rm semi}_0 (X)$ (resp. $G^{\rm semi}_0(X)$) is the quotient by the subgroup generated by the images of the l.c.i. pull-backs $ i_1 ^* - i_2 ^*\colon K_0 (X \times C)\to K_0 (X)$ (resp. $i_1 ^* - i_2 ^* \colon G_0 (X \times C) \to G_0 (X)$) over the equivalence classes of the triples $(C, t_1, t_2)$. When $X$ is smooth, we have 
$K_0 ^{\rm semi} (X) \xrightarrow{\cong} G_0 ^{\rm semi} (X)$. We have the following analogue of \cite[Corollary 4.2.12]{LM}.

\begin{theorem}\label{equivalence with K}\label{thm:Cob-K}
Let $X \in \Sch_k$ and let $\beta$ be a formal symbol of degree $-1$. Then, there is a natural map $\Omega_* ^{\rm alg}(X) \to G_0 ^{\rm semi}(X) [ \beta, \beta^{-1}]$ which induces an isomorphism $\Omega^{\rm alg}_*(X) \otimes_{\mathbb{L}_*} \mathbb{Z}[\beta, \beta^{-1}] \xrightarrow{\simeq} G_0 ^{\rm semi}(X) [\beta, \beta^{-1}].$
\end{theorem} 
\begin{proof}
This follows from the definition of $G^{\rm semi}_0(X)$ above, Theorem \ref{thm:FES-ST}, together with \cite[Corollary 4.2.12]{LM} (if $X$ is smooth) and \cite[Theorem 1.5]{Dai} (if $X$ is not smooth) 
by repeating the arguments in the proof of Theorem \ref{thm:Cob-Chow} in verbatim. The only change is that we have to apply the functor
$ - \otimes_{\mathbb{L}_*} \mathbb{Z} [ \beta, \beta^{-1}]$ instead of 
$- \otimes_{\mathbb{L}_*} \mathbb{Z}$ to the exact sequence similar to that of ~\eqref{eqn:Cob-G}, where the bottom row consists of $G_0$ instead of $\CH$.
\end{proof}


\subsection{Comparison with algebraic cobordism with finite coefficients}\label{subsection:Finite-coeff} 
By \cite[Corollary 3.8]{FW1}, we know that with finite coefficients, the algebraic and the semi-topological $K$-theories of complex projective varieties coincide. The following is the cobordism analogue of this agreement.

\begin{theorem}\label{thm:FC}
Let $X \in \Sch_k$ and let $m \ge 1$ be an integer. Then the natural map $\Phi_X \otimes \mathbb{Z}/m\colon \Omega_*(X)\otimes_{\mathbb{Z}} {\mathbb{Z}}/m \to \Omega^{\rm alg}_*(X)  \otimes_{\mathbb{Z}} {\mathbb{Z}}/m$ is an isomorphism.
\end{theorem}

\begin{proof}
Using \cite[Theorem 1]{LP} and Theorem \ref{thm:comparison}, we can identify $\Omega_*(X)$ and $\Omega^{\rm alg}_*(X)$ with $\omega_*(X)$ and $\omega^{\rm alg}_*(X)$, respectively. In the diagram \eqref{eqn:FES*1}, it suffices to show that ${\rm Im}(\theta)$ in $\omega_*(X)$ is divisible.

Let $(g, p, \zeta)$ be a double-point cobordism with a projective $g \colon Y \to X \times C$, two points $p, \zeta \in C(k)$ and $\pi = pr_2 \circ g$ such that $\pi^{-1} (p) = A \cup B$ 
(see Definition \ref{double point degeneration}). 

Let $\alpha \colon= [Y_\zeta \to Y] - [A \to Y] - [B \to Y] + [\mathbb{P}(\pi) \to Y]$ in $\omega_* (Y)$. Set $f\colon= pr_1 \circ g\colon Y \to X$. Since $\partial_C(g, p, \zeta)= f_*( \alpha) $, it suffices to show that $\alpha$ is divisible in $\omega_*(Y)$. An application of \eqref{adpc equation} to the divisor $E\colon= A + B$ shows that $[A \to Y] + [B \to Y] - [\mathbb{P}(\pi) \to Y] = [E \to Y] =\pi^*([\{p\} \to C])$. We also have $[Y_\zeta \to Y] = \pi^*([\{\zeta\} \to C])$. Thus, $\alpha = \pi^*\left([\{\zeta\} \to C] - [\{p\} \to C]\right)$ and it reduces to proving that the class $\beta \colon= [\{\zeta\} \to C] - [\{p\} \to C]$ is divisible in $\omega_0(C)$. 

By \cite[Lemma 4.5.3]{LM}, the natural map $\omega_0(C) \to \CH_0(C)$ is an isomorphism and the image of $\beta$ in $\CH_0(C)$ is $[\{\zeta\}] - [\{p\}]$, which lies in $\Pic^0(C)$. Since $\Pic^0(C)$ is an abelian variety, the group $\Pic^0 (C)(k)$ is divisible. This completes the proof. 
\end{proof}

\section{Computations of $\Omega_{*} ^{\rm alg}$ and questions on finite generation}\label{sec:computation}
It is usually not easy to compute $\Omega_*$. For the point $X= \Spec (k)$, Levine and Morel \cite[Theorem 1.2.7]{LM} showed that the natural map $\mathbb{L} _*\to \Omega_*(k)$ is an isomorphism. In this section, we focus on some computational aspects of $\Omega_* ^{\rm alg}$.  

\subsection{Comparison with the complex cobordism}\label{subsection:STCC}
We refer to \cite{Quillen} or \cite{Totaro} for the definition and basic properties of the complex cobordism theory $\MU^*$ for locally compact Hausdorff topological spaces. We only mention here that $\MU^*(X)$ is generated by $[f\colon Y \to X]$, where $f$ is proper and $Y$ is a weakly complex real manifold under certain ``bordism relations''.

\begin{proposition}\label{top cycle class sm}\label{prop:STCC}
Given an embedding $\sigma \colon k \hookrightarrow \mathbb{C}$, there is a natural transformation $\theta^{\rm alg}  \colon \Omega_{\rm alg}^* \to \MU^{2*}$ of oriented cohomology theories on $\Sm_k$ that factors the cycle class map $\theta \colon \Omega ^* \to \MU^{2*}$. 
This $\theta^{\rm alg}$ is a lifting of the cycle class map of Totaro $\CH^*_{\rm alg} (X) \to \MU^{2*}(X) \otimes_{\mathbb{L}^*} \mathbb{Z}$ (see \cite[Theorem 3.1]{Totaro}).

\end{proposition} 

\begin{proof}
From \cite[Example 1.2.10]{LM}, we have a morphism $\theta \colon  \Omega ^* \to \MU^{2*}$ of oriented cohomology theories on $\Sm_k$. Hence by Theorem \ref{thm:universal}, it suffices to show that for any $X \in \Sm_k$ and algebraically equivalent line bundles $L_1$ and $L_2$ on $X$, one has $\tilde{c}_1(L_1) = \tilde{c}_1(L_2) \colon \MU^*(X_\sigma) \to \MU^{*+2}(X_\sigma)$. We can assume $k = \mathbb{C}$.

Let $\mathcal{L}$ be a line bundle on $X \times C$ for some compact Riemann surface $C$ such that for some points $t_1, t_2 \in C$, we have $L_j = \mathcal{L}|_{X \times\{ t_j\}}$ for $j = 1,2$. Let $i_j \colon X \times \{t_j \}\to X \times C$ be the inclusions. Take any differentiable path $I$ in $C$, diffeomorphic to the unit interval $[0,1]$, whose end points are $t_1$ and $t_2$. Let $\alpha \colon X \times I \to X \times C$ and $\iota_j \colon X \times \{t_j\} \to X \times I$ be the inclusions. Note that $\alpha \circ \iota_j = i_j$ for $j=1, 2$.

Since $X$ is smooth, we have $ \tilde{c}_1(L_j)([Y \to X]) =  \left(\tilde{c}_1(L_j)(1_X)\right) \cdot [Y \to X] = c_1(L_j) \cdot [Y \to X]$, where the first equality comes from \cite[(5.2)-5]{LM}. On the other hand, we have $c_1(L_j) = i^*_j(c_1(\mathcal{L})) = \iota_j^*\alpha^*(c_1(\mathcal{L}))$. The desired assertion now follows from the fact that $\iota^*_j \colon \MU^*(X \times I) \to \MU^*(X)$ is an isomorphism for $j = 1,2$ because $I$ is contractible.
The second assertion follows from the first assertion and
Theorem~\ref{thm:Cob-Chow}. \end{proof}


\subsection{Points}\label{point case}
\begin{proposition}\label{prop:Cob-point}
The map $\mathbb{L}^* \to \Omega^*_{\rm alg}(k)$ is an isomorphism.
\end{proposition}
\begin{proof}Composing the isomorphism $\mathbb{L}^* \overset{\simeq}{\to} \Omega^* (k)$ with the surjection $\Omega^* (k) \to \Omega^* _{\rm alg} (k)$, we see that the map $\mathbb{L}^* \to \Omega^* _{\rm alg} (k)$ is surjective. We prove injectivity.

We first prove the injectivity of the map $\mathbb{L}^* \otimes _{\mathbb{Z}} \mathbb{Q} \to \Omega^* _{\rm alg} (k) \otimes_{\mathbb{Z}}{\mathbb{Q}}$ with the rational coefficients. Applying Proposition \ref{prop:Alg-st}, we see that $\ker (\Omega^*(k) \to \Omega^*_{\rm alg}(k))$ is generated by the cobordism cycles of the form $\alpha = [Y \to \Spec(k), L]- [Y \to \Spec(k), M]$, where $L \sim M$ on $Y$. Since we are working with the rational coefficients, we can use \cite[Theorem 1, Corollary 3]{LP} to assume that $Y$ is a product of projective spaces. But on such spaces, two lines bundles are algebraically equivalent if and only if they are isomorphic. In particular, $\alpha$ is zero already in $\Omega^*(k)\otimes_{\mathbb{Z}} \mathbb{Q}$. Thus, the map $\mathbb{L} ^*\otimes _{\mathbb{Z}} \mathbb{Q} \to \Omega^* _{\rm alg} (k) \otimes_{\mathbb{Z}}{\mathbb{Q}}$ is injective. The injectivity of $\mathbb{L} ^*\to \Omega^* _{\rm alg} (k)$ now follows because $\mathbb{L}^*$ has no torsion.
\end{proof}

Recall from \cite[Definition 4.4.1]{LM} that an oriented Borel-Moore homology theory $A_*$ on $\Sch_k$ is said to be \emph{generically constant}, if for each finitely generated field extension $k \subset F$, the canonical morphism $A_*(k) \to A_*(F/k)$ is an isomorphism. Here $A_*(F/k)$ is the colimit of $A_{*+ \tr_{F/k}}(X)$ over models $X$ for $F$ over $k$ and $\tr_{F/k}$ is the transcendence degree of $F$ over $k$. Recall that a model for $F$ over $k$ is an integral scheme $X \in \Sch_k$ whose function field is isomorphic to $F$.

\begin{proposition}\label{prop:Gen-const}
The cobordism theory $\Omega^{\rm alg}_*$ is generically constant.
\end{proposition}
\begin{proof}
Let $\mathcal{C}$ denote the category of models for $F$ over $k$. Then, we have a commutative diagram
\[
\xymatrix{
\Omega_*(k) \ar[r]^{\simeq} \ar[d]_{\eta_F} ^{\simeq}  & \Omega^{\rm alg}_*(k) \ar[d]_{\eta^{\rm alg}_F} \ar[dr]^{\simeq} & \\
{\underset{X \in \mathcal{C}}\colim} \ \Omega_{*+ \tr_{F/k}}(X) \ar[r] & {\underset{X \in \mathcal{C}}\colim} \ \Omega^{\rm alg}_{*+ \tr_{F/k}}(X) \ar[r] & \Omega^{\rm alg}_*(F).}
\]

We need to show that $\eta^{\rm alg}_F$ is an isomorphism. It follows from \cite[Corollary 4.4.3]{LM} that $\eta_F$ is an isomorphism. 
Applying Proposition \ref{prop:Alg-st} to the first horizontal arrow on the bottom, we see that $\eta^{\rm alg}_F$ is surjective. On the other hand, it follows from Proposition \ref{prop:Cob-point} that the slanted downward arrow is an isomorphism. This in turn implies that $\eta^{\rm alg}_F$ must also be injective, and hence an isomorphism.
\end{proof}

Recall from \cite[Example 1.9.1]{Fulton} that a scheme $X \in \Sch_k$ is called \emph{cellular} if it has a filtration $\emptyset = X_{n+1} \subsetneq X_n \subsetneq \cdots \subsetneq X_1 \subsetneq X_0 = X$ by closed subschemes such that each $X_i \backslash X_{i+1}$
is a disjoint union of affine spaces, called \emph{cells}. Basic examples include projective spaces, smooth projective toric varieties, and schemes of type $G/P$, where $P$ is a parabolic subgroup of a split reductive group $G$. As a consequence of Proposition \ref{prop:Cob-point} and a general result of A. Nenashev and K. Zainoulline \cite[Theorem 5.9]{NZ}, we can easily compute our cobordism theory for cellular schemes:

\begin{proposition}\label{prop:Cell}
For a cellular scheme $X \in \Sch_k$, the natural map $\Phi_X\colon \Omega_*(X)  \to \Omega^{\rm alg}_*(X)$ is an isomorphism. Each of these groups is a free 
$\mathbb{L}_*$-module of rank equal to the number of cells. 
\end{proposition}



\begin{remark}One may also directly prove Proposition \ref{prop:Cell} by an induction argument on the length of a filtration of $X$ using Theorem \ref{thm:localization**} and \cite[Proposition 4.3]{Krishna}. If there is an embedding $\sigma \colon k \hookrightarrow \mathbb{C}$, Proposition \ref{prop:Cell} can be also deduced from Proposition \ref{prop:Alg-st}, Proposition \ref{top cycle class sm} and \cite[Theorem 6.1]{HK}.\end{remark}

\subsection{Curves}\label{curve case}\label{subsection:curve-case}
We next compute the cobordism theory $\Omega^*_{\rm alg}(X)$ of a smooth curve $X$. We show that this is a finitely generated $\mathbb{L}^*$-module. This is usually false for the algebraic cobordism $\Omega^*(X)$ unless $X$ is rational. If $k=\mathbb{C}$,  we show that $\Omega^*_{\rm alg}(X)$ is closely related to the complex cobordism $\MU^*(X(\mathbb{C}))$. 

\begin{theorem}\label{thm:curve}Let $X$ be a connected smooth curve over a field $k$. Then, 
\begin{enumerate}
\item The $\mathbb{L}^*$-module $\Omega^* _{\rm alg} (X)$ is generated by at most $2$ elements.
\item  When $X$ is affine, the map $\mathbb{L}^* \to \Omega^*_{\rm alg}(X)$ is an isomorphism.
\item  When $k= \mathbb{C}$, 
the map $\Omega ^* _{\rm alg} (X) \to \MU^{2*} (X(\mathbb{C}))$ is an isomorphism.
\end{enumerate}
\end{theorem}
\begin{proof} 
We have shown in Theorem \ref{thm:localization**} and Proposition \ref{prop:Gen-const} that $\Omega_*^{\rm alg}$ has the localization
property and is generically constant. Hence, it satisfies the generalized degree formula \cite[Theorem 4.4.7]{LM}. By this degree formula, the cobordism $\Omega^* _{\rm alg}(X)$ is generated as an $\mathbb{L}^*$-module by the cobordism cycles $1_X =[X \to X]$ and $[\{p\} \to X]= [X \to X, O_X (p)]$, where $p$ is a closed point of $X$. Part (1) now follows from the fact
that the map $\deg\colon{\Pic(X)}/{\sim} \to \mathbb{Z}$ is injective.

If $X$ is affine, we choose a smooth compactification $j \colon X \hookrightarrow \overline{X}$ and set $Z \colon= \overline{X} \backslash X$. This yields an exact sequence
\[
\CH^0(Z) \to {\Pic(\overline{X})}/{\sim} \xrightarrow{j^*} 
{\Pic(X)}/{\sim} \to 0
\]
by \cite[Example 10.3.4]{Fulton}, in which the first map is surjective. In particular, the last term is zero. Thus $\Omega^*_{\rm alg}(X)$ is generated by $1_X$ as an $\mathbb{L}^*$-module, \emph{i.e.}, $\mathbb{L}^* \to \Omega^*_{\rm alg}(X)$ is surjective.
On the other hand, for a closed point $p \in X$, the composition with the pull-back $\mathbb{L}^* \to \Omega^*_{\rm alg}(X) \to
\Omega^*_{\rm alg}\left(k(p)\right)$ is an isomorphism by Proposition \ref{prop:Cob-point}. We conclude that the map $\mathbb{L}^* \to \Omega^* _{\rm alg} (X)$ is injective and hence an isomorphism. This proves (2).

For (3), we first observe that as $X(\mathbb{C})$ is a topological surface, we have an induced isomorphism 
\begin{equation}\label{eqn:curve1}
\MU^* (X (\mathbb{C})) \overset{\simeq}{\to} H^* (X (\mathbb{C}), \mathbb{Z}) \otimes _{\mathbb{Z}} \mathbb{L}^*
\end{equation}
by \cite[Theorem 2.2]{Totaro}.
Since the cycle class map of Proposition \ref{top cycle class sm} maps $\Omega_{\rm alg} ^* (X)$ into $\MU^{2*} (X (\mathbb{C}))$, we look at only the even degrees. When $X$ is affine, we have $H^2 (X(\mathbb{C}), \mathbb{Z}) = 0$ so that $\MU^{2*} (X(\mathbb{C})) = H^0 (X (\mathbb{C}), \mathbb{L}^*) = \mathbb{L}^*$. By part (2), the natural map $\Omega_{\rm alg} ^* (X) \to \MU^{2*} (X (\mathbb{C}))$ is simply the identity map of $\mathbb{L}^*$.

When $X$ is not affine (thus, projective), we get $H^i (X(\mathbb{C}) , \mathbb{L}^*) = \mathbb{L}^*$ for both $i=0$ and $2$, and this yields
$\MU^{2*} (X(\mathbb{C})) \simeq \mathbb{L}^* \oplus \mathbb{L}^*$. Take any closed point $p \in X$ and set $U= X \backslash \{p \}$ (which is affine). 
We get the localization diagram

\begin{equation}\label{eqn:curve2}
\xymatrix{
0 \ar[r] & \Omega^*_{\rm alg}(\{p\}) \ar[r] \ar[d] & \Omega^*_{\rm alg}(X) \ar[r] \ar[d] & \Omega^*_{\rm alg}(U) \ar[r] \ar[d] & 0 \\
0 \ar[r] & \mathbb{L}^* \ar[r] & \mathbb{L}^* \oplus \mathbb{L}^*  \ar[r] & \mathbb{L}^* \ar[r] & 0,}
\end{equation}   
where the bottom exact row is the sequence of $\MU^{2*}$ groups of the spaces $\{p \}$, $X$ and $U$. The top row is exact because the left vertical map is an isomorphism (plus Theorem \ref{thm:localization**}). The right vertical map is an isomorphism because $U$ is affine. Hence, the middle map is an isomorphism too.
\end{proof}


As an immediate corollary of Theorems \ref{thm:FC} and \ref{thm:curve}, we obtain the following analogue of Quillen-Lichtenbaum conjecture for the cobordism of smooth curves.

\begin{corollary}\label{cor:LC-cob}
For a smooth curve $X$ over $\mathbb{C}$ and an integer $m \ge 1$, the natural map $\Omega^*(X) \otimes _\mathbb{Z}{\mathbb{Z}}/m \to \MU^{2*}(X(\mathbb{C})) \otimes_{\mathbb{Z}} {\mathbb{Z}}/m$ is an isomorphism.
\end{corollary}

\subsection{Surfaces}\label{surface case} For an algebraic surface $X$, the Chow groups of $1$-cycles as well as $0$-cycles modulo rational equivalence often form infinitely generated abelian groups. Since the algebraic cobordism contains more data than Chow groups as shown in \cite[Theorem 4.5.1]{LM}, the algebraic cobordism of a surface is often infinitely generated as an $\mathbb{L}^*$-module. However, under algebraic equivalence, the algebraic cycles on an algebraic surface always form a finitely generated group, by N\'eron-Severi theorem.
We prove an analogous result for the $\mathbb{L}^*$-module $\Omega^* _{\rm alg}(X)$. We use the following graded Nakayama lemma whose proof is an elementary application of a backward induction argument. It is left as an exercise.

\begin{lemma}\label{L Nakayama}\label{lem:NAK}
Let $M^*$ be a $\mathbb{Z}$-graded $\mathbb{L}^*$-module such that for some integer $N \geq 0$, we have $M^n = 0$ for all $n > N$. Suppose that $S = \{\alpha_1, \cdots, \alpha_r\}$ is a set of homogeneous elements in $M^{\ge 0}$ whose images generate $M^* \otimes _{\mathbb{L}^*} \mathbb{Z}$ as an abelian group. Then $M^*$ is generated by $S$ as an $\mathbb{L}^*$-module. 
\end{lemma}

\begin{theorem}\label{thm:surface}Let $X$ be a connected smooth projective surface. Then $\Omega^* _{\rm alg} (X)$ is a finitely generated $\mathbb{L}^*$-module with at most $\rho + 2$ generators, where $\rho$ is the minimal number of generators of the N\'eron-Severi group $\NS(X)$.
\end{theorem}
Note that if $\NS(X)$ is torsion free, then $\rho$ is the Picard number of $X$.
\begin{proof}
This follows immediately from Theorem \ref{thm:Cob-Chow} and Lemma \ref{lem:NAK}, using the fact that $\CH^*_{\rm alg}(X) \simeq \mathbb{Z} \oplus \NS(X) \oplus \mathbb{Z}$.
\end{proof}

\subsection{Threefolds and beyond}\label{threefold case} We saw that for a smooth projective scheme $X$ of dimension $\leq 2$, the $\mathbb{L}^*$-module $\Omega^* _{\rm alg} (X)$ is finitely generated. But, this is the highest we can go. This is due to the following result and some known deep results about algebraic cycles. Recall that for a smooth projective complex scheme $X$, the Griffiths group $\Griff^r (X)$ of $X$ is the group of codimension $r$ homologically trivial cycles modulo algebraic equivalence. In particular, it is a subgroup of $\CH^r_{\rm alg}(X)$.

\begin{theorem}\label{finiteness summary}\label{thm:Finiteness}
For any smooth scheme $X$, the following two statements are equivalent:

\emph{(1)} The Chow group $\CH^* _{\rm alg} (X)$ modulo algebraic equivalence is finitely generated.

\emph{(2)} The cobordism $\Omega^* _{\rm alg} (X)$ is a finitely generated $\mathbb{L}^*$-module.

If $X$ is a smooth projective complex variety, then the following statement is also equivalent to the above two:

\emph{(3)} The Griffiths group $\Griff^* (X)$ is finitely generated.




\end{theorem}

\begin{proof} The equivalence (1) $\Leftrightarrow$ (2) follows by
applying Theorem \ref{thm:Cob-Chow} and Lemma \ref{lem:NAK} to 
$M^* = \Omega^*_{\rm alg}(X)$.

When $X$ is a smooth projective complex variety, let $\CH^* _{\rm hom}(X)$ denote the group of algebraic cycles on $X$ modulo homological equivalence. The equivalence (1) $\Leftrightarrow$ (3) follows from the exact sequence
\begin{equation}\label{eqn:Griff}
0 \to \Griff^* (X) \to \CH^* _{\rm alg} (X) \to \CH^* _{\rm hom} (X) \to 0
\end{equation} 
and the observation that $\CH^* _{\rm hom}(X)$ is a subgroup of $H^{2*} (X(\mathbb{C}), \mathbb{Z})$, which is a finitely generated abelian group since $X$ is smooth and projective.\end{proof}



\begin{remark}It was shown by Griffiths \cite{Griff} that the Griffiths groups can be nontrivial. Clemens \cite{Clemens} later showed that $\Griff^2(X)$ is not finitely generated for a general quintic threefold $X$. These results were generalized by Nori \cite{Nori} for algebraic cycles of codimension $\geq 2$. Thus, it follows from Theorem \ref{thm:Finiteness} that the $\mathbb{L}^*$-module $\Omega^*_{\rm alg}(X)$ is in general not finitely generated for a variety of dimension at least three.
\end{remark}

It seems that certain questions about algebraic cycles of smooth projective 
schemes can be lifted to the level of cobordism cycles. As an example, consider the following. We saw in Section \ref{subsection:STCC} that for a smooth complex variety $X$, there are cycle class maps $\theta_X \colon \Omega^*(X) \to \MU^{2*}(X(\mathbb{C}))$ and $\theta^{\rm alg}_X \colon\Omega^*_{\rm alg}(X) \to \MU^{2*}(X(\mathbb{C}))$. Let $\Phi_X \colon \Omega^*(X) \to \Omega^*_{\rm alg}(X)$ be the natural map. We define the \emph{Griffiths groups for the cobordism cycles} to be the graded group
\begin{equation}\label{eqn:Griff-Cob}
\Griff^*_{\Omega} (X) = {\ker(\theta_X)}/{\ker(\Phi_X)}.
\end{equation}
The subgroup $\ker(\theta_X)$ can be called the group of cobordism cycles \emph{homologically equivalent to zero}. We ask the following:

\begin{question}\label{ques:Cob-Griif-fin}
Let $X$ be a smooth projective complex variety of dimension at least three. Is it true that $\Griff^*_{\Omega} (X)$ is a finitely generated $\mathbb{L}^*$-module if and only if $\Griff^*(X)$ is a finitely generated abelian group? In particular, are there examples where $\Griff^*_{\Omega} (X)$ is not finitely generated as an $\mathbb{L}^*$-module?
\end{question}

\section{Rational smash-nilpotence for cobordism}\label{sec:smash-nil}
It was proven by Voevodsky \cite{Voevodsky} and Voisin \cite{Voisin} that if an algebraic cycle $\alpha$ on a smooth projective scheme $X$ is zero in $\CH_*^{\rm alg} (X)_{\mathbb{Q}}$, then the smash-product $\alpha^{\otimes N}\colon= \alpha \times \cdots \times \alpha$ on $X^N \colon= X \times \cdots \times X$ is zero in $\CH_* (X^N)_{\mathbb{Q}}$ for some integer $N>0$. We use the notation $\alpha^{\otimes N}$ instead of $\alpha^N$. The latter symbol denotes the self-intersection of $\alpha$ in $\CH_* (X)_{\mathbb{Q}}$. This section studies the corresponding question for cobordism cycles. 

\begin{definition}\label{def:smash-product}Let $X \in \Sch_k$ and let $\alpha \in \mathcal{Z}_* (X)$. Let $N \geq1$ be an integer.

(1) The \emph{$N$-fold smash-product} $\alpha ^{\boxtimes N} \in \mathcal{Z}_* (X^N)$ is the $N$-fold self-external product 
$\alpha \times \cdots \times \alpha$ (see Definition \ref{basic definitions for cobordism cycles}). 


(2) $\alpha$ is \emph{rationally smash-nilpotent}, if there is an integer $N>0$ such that the image of $\alpha^{\boxtimes N}$ in $\Omega_* (X^N)_{\mathbb{Q}}$ is zero.
\end{definition}

\begin{lemma}\label{prep nilp}Let $X \in \Sch_k$ and let $\alpha, \beta \in \mathcal{Z}_* (X)$.

\emph{(1)} If $\alpha$ or $\beta$ is rationally smash-nilpotent, then so is $\alpha \times \beta$. 

\emph{(2)} If $\alpha$ and $\beta$ are rationally smash-nilpotent, then so is $\alpha + \beta$.
\end{lemma}

\begin{proof} Note that the external product $\times$ is commutative because in Definition \ref{cobordism cycle LM} we identified all isomorphic cobordism cycles. For (1), if $\alpha^{\boxtimes N} = 0 \in \Omega_* (X^N)_{\mathbb{Q}}$, then $(\alpha \times \beta)^{\boxtimes N} = \alpha^{\boxtimes N} \times \beta^{\boxtimes N} = 0 \in \Omega_* (X^{2N})_{\mathbb{Q}}$. The case $\beta ^{\boxtimes N} = 0$ in $\Omega_* (X^N) _{\mathbb{Q}}$ is similar. For (2), use the binomial theorem since $\times$ is commutative.
\end{proof}


We now prove the cobordism analogue of the result \cite[Corollary 3.2]{Voevodsky}:

\begin{theorem}\label{thm:smash nil}Let $X$ be a smooth projective scheme and
let $\alpha \in \mathcal{Z}_* (X)$. If the image of $\alpha$ in $\Omega_* ^{\rm alg} (X) _{\mathbb{Q}}$ is trivial, then it is rationally smash-nilpotent.
\end{theorem}

\begin{proof}By \cite[Theorem 1]{LP} and Theorem \ref{thm:comparison}, we may identify $\Omega_* (X)$ and $\Omega_* ^{\rm alg} (X)$ with $\omega_* (X)$ and $\omega_* ^{\rm alg} (X)$, respectively. Consider the exact sequence of Theorem \ref{thm:FES} with rational coefficients,
\[
{\underset{(C,t_1, t_2)}\bigoplus} \omega_*(X \times C)_{\mathbb{Q}} \xrightarrow{\theta'} \omega_*(X)_{\mathbb{Q}} \xrightarrow{\Psi_X} \omega^{\rm alg}_*(X)_{\mathbb{Q}} \to 0
\]
and look at the image of $\alpha$ in $\omega_* (X)_{\mathbb{Q}}$, also denoted by $\alpha$. Since $\alpha$ belongs to $ \ker \Psi_X$ by the given assumption, we have $\alpha \in {\rm Im} ( \theta') $. By Lemma \ref{prep nilp}-(2), it is enough to consider $\alpha$ of the form $(i_1 ^* - i_2 ^* ) (\beta)$ for $\beta = [g\colon Y \to X\times C]\in \omega_* (X \times C)$. So, we suppose $\alpha=(i_1 ^* - i_2 ^*) (\beta)$. 

Since $X \times C$ is smooth, by the transversality \cite[Proposition 3.3.1]{LM} combined with \cite[Theorem 1]{LP}, we may assume that $g$ is transversal to the closed immersions $i_j$, $j=1,2$. Hence, the fiber product $Y_{t_j}$ of $X\times \{ t_j\}$ and $Y$ over $X \times C$ is smooth and $i_j ^* (\beta) = [Y_{t_j} \to X]$. On the other hand, if we let $\pi \colon= pr_2 \circ g \colon Y \to X \times C \to C$ (which is projective), we see that
$\pi^* \left([ \{t_j\} \to C ]\right) = [Y_{t_j} \to Y]$. Hence, for $f\colon= pr_1 \circ g \colon Y \to X \times C \to X$ (which is projective because
$C$ is projective), we get $f_* \pi^* \left([ \{t_j\} \to C]\right) = [Y_{t_j} \to X]$ so that $ (i_1 ^* - i_2 ^*) (\beta) = [ Y_{t_1} \to X] - [Y _{t_2} \to X] = f_* \pi^* ( [ \{t_1\} \to C] - [\{t_2\} \to C]).$ Set $\gamma\colon= [ \{t_1\} \to C] - [\{t_2\} \to C] \in \omega_0 (C) _{\mathbb{Q}}$. 

We then have $\alpha = f_* \pi^* (\gamma)$ with $\gamma \in \omega_0 (C)_{\mathbb{Q}}$ such that $\gamma = 0 \in \omega_0 ^{\rm alg} (C)_{\mathbb{Q}}$. We claim that $\gamma$ is rationally smash-nilpotent. 

Under the isomorphism $\omega_0  (C) _{\mathbb{Q}}  \overset{\simeq}{\to} \CH_0 (C)_{\mathbb{Q}}$ of \cite[Lemma 4.5.10]{LM}, the image of $\gamma$ in $\CH_0 (C)_{\mathbb{Q}}$ is the $0$-cycle $\bar{\gamma}=[\{t_1\}] - [ \{t_2\} ] \in \CH_0 (C)_{\mathbb{Q}}$, whose image in $\CH_0 ^{\rm alg} (C)_{\mathbb{Q}}$ is trivial. Hence by \cite[Corollary 3.2]{Voevodsky}, we see that 
$\bar{\gamma} ^{\otimes N} = 0 \in \CH_0 (C^N)_{\mathbb{Q}}$ for some integer $N>0$. Since the isomorphism $\omega_0 (C^N)_{\mathbb{Q}} \simeq \CH_0 (C^N)_{\mathbb{Q}}$ of \cite[Lemma 4.5.10]{LM} respects the external products, we conclude that $\gamma ^{\boxtimes N} = 0 \in \omega_0 (C^N)_{\mathbb{Q}}$.

Since $\gamma$ is rationally smash-nilpotent, we now easily see that $\alpha = f_* \pi^* (\gamma)$ is also rationally smash-nilpotent since the push-forward and the pull-back maps respect external products (\emph{cf.} Theorem \ref{thm:lci prop}).\end{proof}

\begin{remark}We remark that the proof of Theorem \ref{thm:smash nil} uses \cite{Voevodsky} only for smooth projective curves. \end{remark}

\begin{remark}
Theorem \ref{thm:smash nil} shows that all algebraically trivial cobordism cycles on smooth projective schemes are smash-nilpotent with the $\mathbb{Q}$-coefficients. Motivated by \cite[Conjecture 4.2]{Voevodsky}, one can further ask whether numerical triviality of cobordism cycles is equivalent to smash-nilpotence for a suitable notion of numerical equivalence for cobordism cycles. The authors do not know how to answer this.

As a weaker version, we wonder if homologically trivial cobordism cycles are smash-nilpotent, where homological equivalence on cobordism cycles was defined around \eqref{eqn:Griff-Cob}. For abelian varieties one might try the following, motivated by \cite{KS}. Let $A$ be an abelian variety and for each $m \in \mathbb{Z}$, let $\left< m \right> : A \to A$ be the multiplication morphism by $m$. Let's call a cobordism cycle $\beta \in \Omega^* (A)_{\mathbb{Q}}$ \emph{skew} if $\left<-1\right>^* (\beta) = - \beta$. A skew cobordism cycle is homologically trivial. Our guess is that any skew cobordism cycle on an abelian variety $A$ is smash-nilpotent. The corresponding question for algebraic cycles was answered in \cite[Proposition 1]{KS} and it was deduced that any homologically trivial cycle on an abelian variety of dimension $\leq 3$ is smash-nilpotent. 

To answer it, the following strategy is likely to work. Firstly, use the category of cobordism motives in the sense of \cite[\S 5.1]{NZ} and \cite[\S 2]{VY}, a cobordism analogue of the category of Chow motives. Secondly, use the abelian structure on $A$ to prove an analogue of Chow-K\"unneth decomposition. Lastly, imitate the arguments of \cite{KS}. A detailed discussion of this approach
will appear in a separate paper.
\end{remark}

\section{Appendix}\label{section:appendix}
This section gives a summary of the constructions from \cite[\S 6]{LM} related to the Gysin maps and the pull-backs via l.c.i. morphisms on the algebraic cobordism, that are used in this paper. The only new result is Lemma \ref{lem:Int-ps-div}, used in the construction of the map $i_D ^* \colon \Omega_* ^{\rm alg} (X) \to \Omega_{*-1} ^{\rm alg} (|D|)$ of \eqref{eqn:Int-div-st} for a pseudo-divisor $D$ on $X$.

\begin{definition}[{\rm \cite[6.1.2]{LM}}]\label{defn:cobord cycle with D} Let $X \in \Sch_k$ and let $D$ be a pseudo-divisor on $X$. 

(1) $\mathcal{Z}_*(X)_D$ is the subgroup of $\mathcal{Z}_*(X)$ generated by the cobordism cycles $[f\colon Y \to X, L_1, \cdots, L_r]$ for which either $f(Y) \subset |D|$ holds, or $f(Y) \not \subset |D|$ and $\Div f^* D$ is a strict normal crossing divisor on $Y$. 

(2) Let $ \mathcal{R} _* ^{\rm Dim} (X)_D$ be the subgroup of $\mathcal{Z}(X)_D$ generated by the cobordism cycles of the form $[f \colon Y \to X, \pi^* (L_1), \cdots, \pi^* (L_r), M_1, \cdots, M_s]$, where $\pi\colon Y \to Z$ is smooth quasi-projective, $Z \in \Sm_k$, and $L_1, \cdots, L_{r>\dim Z}$ are line bundles on $Z$. We let $\underline{\mathcal{Z}}_* (X)_D\colon= \mathcal{Z}_* (X)_D/ \mathcal{R}_* ^{\rm Dim} (X)_D$. 

The projective push-forward and smooth pull-back on $\mathcal{Z}_* (-)_D$ can be defined as for $\mathcal{Z}_* (-)$, and likewise for $\underline{\mathcal{Z}}_* (-)$. 

(3) For a line bundle $L \to X$, define the Chern class operation $\tilde{c}_1 (L) \colon \mathcal{Z} _* (X)_D \to \mathcal{Z}_{*-1} (X)_D$ as for $\mathcal{Z}_* (X)$. This descends onto $\underline{\mathcal{Z}}_* (X)_D$. 

(4) We have the external product 
\[
\times \colon \mathcal{Z}_* (X) _D \otimes \mathcal{Z}_* (X')_{D'} \to \mathcal{Z}_* (X \times X') _{pr_1 ^* D + pr_2 ^* D'}
\]
as for $\mathcal{Z}_* (-)$. This descends onto $\underline{\mathcal{Z}}_* (-)_D$-level.
\end{definition}

Let $X \in \Sch_k$, $D$ be a pseudo-divisor on $X$, and $f$ be a projective morphism $f\colon Y \to X$ from a smooth irreducible scheme $Y$. A strict normal crossing divisor $E$ on $Y$ is said to be in \emph{very good position with $D$} if either $f(Y) \subset |D|$ holds, or $f(Y) \not \subset |D|$ and $E + \Div f^*D$ is a strict normal crossing divisor on $Y$. If $E$ is in very good position with $D$, then for each face $i_J \colon E_J\hookrightarrow E$ and the induced composition $f_J\colon= f \circ i_J\colon E_J \to Y \to X$, either $f_J (E_J) \subset |D|$ holds, or $\Div f_J ^* D$ is a strict normal crossing divisor on $E_J$, by \cite[Remark 6.1.4(1)]{LM}

\begin{definition}[{\rm \cite[Definition 6.1.5]{LM}}]\label{defn:cobord cycle with D 1} Let $X \in \Sch_k$ and let $D$ be a pseudo-divisor on $X$. Let $\mathcal{R}_* ^{\rm Sect}(X)_D$ be the subgroup of $\mathcal{Z}_* (X)_D$ generated by elements of the form $[f \colon Y \to X, L_1, \cdots, L_r] - [f \circ i\colon Z \to X, i^* (L_1), \cdots, i^* (L_{r-1})],$ with $r>0$, such that 

(1) $[f\colon Y \to X, L_1, \cdots, L_r] \in \mathcal{Z}_* (X)_D$, and

(2) $i\colon Z \to Y$ is the closed immersion of the subscheme given by the vanishing of a transverse section $s\colon Y \to L_r$ such that $Z$ is in very good position with $D$.

We let $\underline{\Omega}_* (X)_D \colon= \underline{\mathcal{Z}}_* (X)_D / \mathcal{R}_* ^{\rm Sect} (X)_D.$

\end{definition}

\begin{definition}[{\cite[Definitions 6.1.6]{LM}}]\label{defn:cobord cycle with D 2}Let $X\in \Sch_k$ and let $D$ be a pseudo-divisor on $X$. 

(1) Let $\mathcal{R}_*(X)_D$ be the subgroup of $\mathcal{Z}_* (X)_D$ generated by elements of the form $[Y \to X, L_1, \cdots, L_r] - [Y' \to X, L_1 ', \cdots, L_r ']$ such that
	\begin{enumerate}
	\item [(a)] $[Y \to X, L_1, \cdots, L_r]$ and $[Y' \to X, L_1 ', \cdots, L_r ']$ are in $\mathcal{Z}_* (X)_D$, and 
	\item [(b)] there exist an isomorphism $\phi\colon Y \to Y'$ over $X$, a permutation $\sigma$ of $\{ 1, \cdots, r\}$ and an isomorphism $L_i \simeq \phi^* (L' _{\sigma (i)})$.
	\end{enumerate}
We define $\underline{\Omega}_* (X)_D\colon= \underline{\Omega}_* (X)_D / \mathcal{R}_*(X)_D$.

(2) Let $\Omega_*(X)_D$ be the quotient of $\mathbb{L}_* \otimes _{\mathbb{Z}} \underline{\Omega}_* (X)_D$ by the relations of the form 
\[
({\rm Id}_{\mathbb{L}_*} \otimes f_*) ( F_{\mathbb{L}_*} (\tilde{c}_1 (L), \tilde{c}_1 (M)) (\eta) - \tilde{c}_1 (L \otimes M) (\eta))
\]
where $L, M$ are line bundles on $Y$, $\eta \in \underline{\Omega}_*(Y)_D$
and $[f \colon Y \to X]$ is a cobordism cycle for which either $f(Y) \subset |D|$ holds, or $f(Y) \not \subset |D|$ and $\Div f^* D$ is a strict normal crossing divisor on $Y$.

The Chern class operation and the external product are induced on $\Omega_*(-)_D$. 
\end{definition}

It is clear from the above definition that there is a natural map $\theta_X \colon \Omega_*(X)_D \to \Omega_*(X)$. The main content of \cite[\S 6.4.1]{LM} is the proof of the following moving lemma.

\begin{theorem}[{\rm \cite[Theorem 6.4.12]{LM}}]\label{thm:CML}
For $X \in \Sch_k$, the natural map $\theta_X\colon \Omega_* (X)_D \to \Omega_* (X)$ is an isomorphism.
\end{theorem}

Now we define the intersection by $D$ on $\Omega_*  (X)_D$, namely, $D(-) \colon \Omega_*(X)_D \to \Omega_{*-1}  (|D|)$. First recall the map $D(-) \colon \mathcal{Z}_* (X)_D \to \Omega_{*-1} (|D|)$.

\begin{definition}[{\rm \cite[\S 6.2.1]{LM}}]\label{defn:Intersection*}
Let $X \in \Sch_k$ and let $D= (|D|, O_X (D), s)$ be a pseudo-divisor on $X$. Let $\eta \colon = [ f \colon Y \to X, L_1, \cdots, L_r] \in \mathcal{Z}_* (X)_D$.

(1) If $f(Y) \subset |D|$, let $f^D \colon Y \to |D|$ be the induced morphism from $f$. Note that $\tilde{c}_1 (f^* O_X (D)) ([{\rm Id}_Y \colon Y \to Y, L_1, \cdots, L_r]) \in \Omega_{*-1} (Y)$. We define 
\[
D(\eta) \colon= f^D _* \left\{ \tilde{c}_1 (f^* O_X (D))\left([{\rm Id}_Y\colon Y \to Y, L_1, \cdots, L_r]\right) \right\}\in \Omega_{*-1} (|D|).
\] 

(2)  If $f(Y) \not \subset |D|$, then $\tilde{D}\colon= \Div f^* D$ is a strict normal crossing divisor on $Y$. Let $f^D \colon |\tilde{D}| \to |D|$ be the restriction of $f$, and $L_i ^D$ be the restriction of $L_i$ on $|\tilde{D}|$. We define $D(\eta)\colon=f_* ^D \{ \tilde{c}_1 (L_1 ^D) \circ \cdots \circ \tilde{c}_1 (L_r ^D)  ( [\tilde{D} \to |\tilde{D}|])\} \in \Omega_{*-1}(|D|),$ where the cobordism cycle $[\tilde{D} \to |\tilde{D}|] \in \Omega_* (|\tilde{D}|)$ is discussed in Section \ref{background for Omega adpc} and \cite[\S 3.1]{LM}.
\end{definition}

This descends to give $D(-) \colon \Omega_* (X)_D \to \Omega_{*-1} (|D|)$ by \cite[\S 6.2]{LM}. To show that it induces \eqref{eqn:Int-div-st}, we need the following:

\begin{lemma}\label{lem:Int-ps-div}
Let $X \in \Sch_k$ and let $D$ be a pseudo-divisor on $X$. Let $[f\colon Y \to X, L]$ and $[f\colon Y \to X, M]$ be two cobordism cycles in $\Omega_*(X)$ such that $L\sim M$. Let $\eta\colon=[Y \to X, L] - [Y \to X, M]$. Then $D\circ \phi_X (\eta) \in \ker(\Omega_{*-1}(|D|) \to \Omega^{\rm alg}_{*-1}(|D|))$, where $\phi_X = \theta_X ^{-1}$ of Theorem \ref{thm:CML}.
\end{lemma}

\begin{proof}
We first assume that $[f\colon Y \to X, L]$ and $[f\colon Y \to X, M]$ lie in $\Omega_*(X)_D$ and show that $D(\eta) \in \ker(\Omega_{*-1}(|D|) \to \Omega^{\rm alg}_{*-1}(|D|))$. By the definition of $D(-)$ in Definition \ref{defn:Intersection*}, there is nothing to prove if $f(Y) \subset |D|$. So, suppose $f(Y) \not \subset |D|$. Then, we have 
\begin{equation}\label{eqn:Int-ps-div1}
D([f\colon Y \to X, L]) = f^D_*\{\tilde{c}_1(L|_{\widetilde{D}})([\widetilde{D} \to |D|])\} \in \Omega_{*-1}(|D|)
\end{equation}
and the similar expression holds for $D([f\colon Y \to X, M])$. On the other hand, $L \sim M$ implies that
$L|_{\widetilde{D}} \sim M|_{\widetilde{D}}$ and hence $\tilde{c}_1(L|_{\widetilde{D}})  = \tilde{c}_1(L|_{\widetilde{D}})$ as operators on $\Omega^{\rm alg}_*(|\widetilde{D}|)$. Using Proposition \ref{prop:universal lev0} and applying $f^D_*$, we get
from \eqref{eqn:Int-ps-div1} that $D([f\colon Y \to X, L]) =  D([f\colon Y \to X, M])$ in $\Omega^{\rm alg}_{*-1}(|D|)$. Equivalently, $D(\eta)  \in \ker \left(\Omega_{*-1}(|D|) \to \Omega^{\rm alg}_{*-1}(|D|)\right)$.

To complete the proof, choose a suitable projective birational map $\rho\colon W \to Y \times \mathbb{P}^1$ as in \cite[Lemma 6.4.1]{LM}. This yields a 
commutative diagram 
\begin{equation}\label{eqn:Int-ps-div2}
\xymatrix{
\Omega_*(W)_{D_{W}} \ar[r]^{D_W(-)} \ar[d]_{(f\circ pr_1 \circ \rho)_*} & \Omega_{*-1}(|D_W|) \ar[d]^{(f\circ pr_1 \circ \rho)_*} \\
\Omega_*(X)_{D} \ar[r]^{D(-)} & \Omega_{*-1}(|D|)}
\end{equation}
where $D_W = \rho^* \circ pr^*_1 \circ f^*(D)$ such that
\[
D \circ \phi_X([Y \to X, L]) = 
(f\circ pr_1 \circ \rho)_* \circ D_W \left([\rho^*(Y \times \{0\}) \to W,
\rho^*(L)]\right).
\]
A similar formula holds for $D \circ \phi_X([Y \to X, M])$. The lemma follows from this by applying what we have shown above to the pair $(W,D_W)$ in place of $(X,D)$.
\end{proof}


\begin{ack}
The authors feel very grateful to Marc Levine and the anonymous referee of the Journal of $K$-theory for useful comments on various parts of this work and for suggesting further research directions. 

JP would like to thank Juya and Damy for the constant supports and peace of mind at home during the work. He also wishes to acknowledge that part of this work was written during his visits to TIFR and the Universit\"at Duisburg-Essen in 2011. He would like to thank all those institutions that made this work possible. 

During this work, JP was partially supported by the National Research Foundation of Korea (NRF) grant (No. 2011-0001182) and Korea Institute for Advanced Study (KIAS) grant, both funded by the Korean government (MEST), and the TJ Park Junior Faculty Fellowship funded by POSCO TJ Park Foundation.

\end{ack}

\end{document}